\documentclass[10pt,a4paper]{amsart}
\usepackage[margin=22mm]{geometry}

\usepackage[T1]{fontenc}
\usepackage[utf8]{inputenc}
\usepackage{amsmath,amssymb,amsthm}
\usepackage[round]{natbib}
\usepackage{graphicx}
\usepackage{subcaption}
\usepackage{booktabs}
\usepackage[scr=rsfs]{mathalfa}
\usepackage{stmaryrd}
\usepackage[colorlinks=true,linkcolor=blue,urlcolor=blue,citecolor=blue,anchorcolor=blue]{hyperref}

\allowdisplaybreaks

\newtheorem{thm}{Theorem}[section]
\newtheorem{prop}[thm]{Proposition}
\newtheorem{lem}[thm]{Lemma}
\newtheorem{cor}[thm]{Corollary}
\theoremstyle{definition}
\newtheorem{defn}[thm]{Definition}
\theoremstyle{remark}
\newtheorem{rem}[thm]{Remark}

\newcommand{\E}{\mathbb{E}}
\renewcommand{\P}{\mathbb{P}}
\newcommand{\V}{\mathrm{Var}}
\newcommand{\Cov}{\mathrm{Cov}}

\newcommand{\tr}{\mathrm{tr}}

\newcommand{\R}{\mathbb{R}}

\newcommand{\independent}{\mathrel{\perp\mkern-10mu\perp}}

\renewcommand{\i}{\imath}

\newcommand{\Gw}{\mathsf{G}}
\newcommand{\Law}{\mathscr{L}}

\newcommand{\F}{\mathcal{F}}

\newcommand{\W}{\mathcal{W}}

\newcommand{\bs}[1]{\boldsymbol{#1}}

\newcommand{\RvS}[1]{}
\newcommand{\BG}[1]{}
\newcommand{\YS}[1]{}
\newcommand{\FRap}[1]{}

\begin{document}

\title{Explicit Gamma--Stein Bounds for Satterthwaite's Approximation}

\author{Gabriel Bailly}\address{LIDAM, UCLouvain, Voie du Roman Pays 20, B-1348 Louvain-la-Neuve, Belgium}
\author{Ferdinand Rapin}\address{D\'epartement de math\'ematiques, Universit\'e libre de Bruxelles, Boulevard du Triomphe, B-1050 Brussels, Belgium}
\author{Yvik Swan}\address{D\'epartement de math\'ematiques, Universit\'e libre de Bruxelles, Boulevard du Triomphe, B-1050 Brussels, Belgium; Vrije Universiteit Brussel, Pleinlaan 2, B-1050 Brussels, Belgium. E-mail: yvik.swan@ulb.be}
\author{Rainer von Sachs}\address{LIDAM, UCLouvain, Voie du Roman Pays 20, B-1348 Louvain-la-Neuve, Belgium}

\date{}

\begin{abstract}
  { We use Stein's method to derive explicit error bounds for the Satterthwaite
  approximation of weighted sums of independent chi-square random
  variables and, more generally, of independent Gamma random
  variables. Our approach relies on new kernel
  representations of the solution to the Gamma--Stein equation and of
  its first and second derivatives on the positive real line which yield weighted Stein-factor bounds adapted to the 
  moment-matching structure of the underlying problem.  In the Zolotarev ideal metrics of orders 1 and 2, we bound the approximation error directly in terms of the differences between the individual scale parameters and the scale parameter of the moment-matched gamma distribution. Smoothing then yields bounds in Kolmogorov distance.
  We apply these results to pooled variances, sample variances of stationary Gaussian time series, traces of Wishart matrices and sums of such traces, and high-dimensional repeated-measures statistics.}
\end{abstract}



\maketitle

\section{Introduction}\label{sec:intro}

Satterthwaite's approximation is a classical moment-matching procedure for approximating the distribution of a weighted sum of independent 
chi-square random variables by a scaled chi-square distribution. Its
origins go back to \citet{Welch1938}, in the context of the two-sample
\(t\)-test with unequal population variances. In that setting, Welch
proposed to replace the exact distribution of the weighted sum
appearing in the test statistic by a chi-square distribution whose
degrees of freedom are chosen to reproduce the first two moments. This
idea was subsequently generalized by
\citet{Satterthwaite1941,Satterthwaite1946} to arbitrary weighted sums
of independent chi-square random variables. In modern terminology, the
approximation consists in replacing such a weighted sum by a scaled
chi-square random variable, with the scale and the degrees of freedom
chosen so that the approximating distribution has the same mean and
variance as the original statistic.

The appeal of this approximation is its simplicity. Exact formulas for
the density of weighted chi-square sums are available, but they
typically involve infinite series or integral representations and are
therefore not always convenient for statistical practice; see, for
instance, \cite{Moschopoulos1985}. Beyond the two-moment Satterthwaite
approximation, several higher-order moment-matching procedures have
been proposed, including the Hall--Buckley--Eagleson, Wood \(F\), and
Lindsay--Pilla--Basak methods; see \cite{Bodenham2015} for a
comparative review. Nevertheless, the original Satterthwaite
approximation remains widely used because it is explicit,
interpretable, and easy to implement.

Satterthwaite-type approximations are especially useful when the exact
null distribution is available in principle but too complicated to
serve as a practical reference distribution. This occurs for many
quadratic-form statistics, whose laws can be expressed as weighted
sums of independent chi-square, or more generally Gamma, random
variables. The examples developed below illustrate such situations:
they include pooled variances, sample variances of stationary Gaussian
time series, traces of Wishart random matrices, sums of independent
Wishart traces, and high-dimensional repeated-measures
designs. Together, these examples cover classical variance estimation,
dependence, multivariate covariance analysis, and modern
high-dimensional inference, illustrating why such approximations
remain important in statistical practice.

Despite this broad use, quantitative theoretical guarantees for
Satterthwaite's approximation remain relatively limited. Classical approaches { to these questions}
often work directly with distribution functions or through asymptotic
expansions. The question addressed here is slightly different: how close is
the law of a weighted chi-square, or more generally a weighted Gamma sum, to
its moment-matched Gamma approximation when both variables are kept on their
original scale? This is the formulation that appears naturally in statistical
applications, where the Satterthwaite approximation is used directly as a
reference distribution.

Stein's method provides a natural framework for this problem. Gamma
approximation by Stein's method goes back at least to \citet{Luk1994}, and
was subsequently developed in several directions, including chi-square
approximation via Stein's method \citep{Pickett2004}, Stein-factor estimates
for Gamma approximation \citep{Gaunt2017,Dobler2018}, and general
density-based constructions of Stein operators
\citep{ley2017distances,ley2017stein}. The work of \citet{Dobler2018} also
belongs to the Malliavin--Stein line of research on Gamma approximation,
where Gamma or centered-Gamma Stein equations are used to prove non-central
limit theorems for nonlinear functionals. These works provide powerful
general tools, while the present paper focuses specifically on the
moment-matching structure of Satterthwaite's approximation for positive
weighted Gamma sums.

A related approach, closer in spirit to our setting, was proposed by
\citet{LiuXia2021}, who use a characterization of the Gamma distribution
through size biasing and zero biasing. Their results yield bounds for Gamma
approximation of sums of independent nonnegative random variables, and in
particular their Corollary~3.4 applies to convolutions of independent Gamma
random variables. Our aim is complementary: we seek bounds stated directly in
terms of the scale discrepancies between the summands and the
moment-matched Gamma approximation.

We now introduce the notation used throughout the paper. We write \(X\cong Y\)
when \(X\) and \(Y\) have the same distribution, that is,
\(\Law(X)=\Law(Y)\), and we write \(X\independent Y\) for independence. The
notation
\[
  (X_i)_{i=1}^n \overset{\mathrm{iid}}{\sim} P
\]
means that \(X_i\sim P\) for every \(i\in\{1,\ldots,n\}\) and that
\(X_1,\ldots,X_n\) are mutually independent.

For \(a,b>0\), we denote by \(\mathcal G_{a,b}\) the Gamma distribution with
shape \(a\) and scale \(b\), namely the distribution with density
\[
  p_{a,b}(x)
  =
  \frac{1}{\Gamma(a)b^a}x^{a-1}e^{-x/b}
  \mathbf{1}_{\{x>0\}} .
\]
We write \(\mathcal Z_1\) for the class of continuously differentiable
functions \(h:(0,\infty)\to\mathbb R\) such that \(\|h'\|_\infty\le1\), where $\|\cdot\|_\infty$ denotes the usual supremum norm. For
positive random variables \(X\) and \(Y\) with finite first moments, define
\[
  d_1(X,Y)
  :=
  \sup_{h\in\mathcal Z_1}
  \left|
    \E[h(X)]-\E[h(Y)]
  \right|.
\]
This is the usual Wasserstein distance of order one, written here in the
smooth test-function formulation that is convenient for Stein's method.

We also write \(\mathcal Z_2\) for the class of twice continuously differentiable
functions \(h:(0,\infty)\to\mathbb R\)  such that \(\|h''\|_\infty\le1\). For positive random variables \(X\) and
\(Y\) with finite second moments, set
\[
  d_2(X,Y)
  :=
  \sup_{h\in\mathcal Z_2}
  \left|
    \E[h(X)]-\E[h(Y)]
  \right|.
\]
This is the Zolotarev ideal metric of order two (see \cite{zolotarev1979ideal}), again expressed in the convenient smooth test-function formulation. It is finite in the natural
setting of this paper, where the two variables have matching first two
moments.

{
The final distance we consider is the Kolmogorov distance
\[
  d_{\mathrm{Kol}}(X,Y)
  :=
  \sup_{z\in\mathbb R}
  \left|
    \P[X\le z]-\P[Y\le z]
  \right|.
\]
We control this distance by smoothing the indicator functions
$\mathbf{1}_{(-\infty,z]}$ and applying our bounds for $d_1(X,Y)$.
The smoothing inequality depends on the shape parameter $a$
of the Gamma target~$Y$.

When $a\ge 1$, the density $p_Y$ is bounded, and we use the classical
smoothing inequality (see, e.g., \citet[p.~48]{Chen2011})
\begin{equation}\label{kolmosmoo}
  d_{\mathrm{Kol}}(X,Y)
  \le
  \sqrt{2B_Y\,d_1(X,Y)},
  \qquad
  B_Y:=\sup_{x\in\mathbb R}p_Y(x)<\infty.
\end{equation}
When $0<a<1$, the Gamma density is unbounded near zero.
For this range, Proposition~\ref{prop:Kol-small-shape} gives a smoothing
inequality that controls $d_{\mathrm{Kol}}(X,Y)$ in terms of $d_1(X,Y)$.
}

These three distances play complementary roles. The Kolmogorov distance is
closest to classical testing, since it controls distribution functions,
critical values, and \(p\)-values. The Wasserstein distance \(d_1\) controls
expectations of Lipschitz transformations and provides, through smoothing, a
route to Kolmogorov bounds. The distance \(d_2\) is especially natural for
Satterthwaite's approximation: since the approximating Gamma law is chosen by
matching the first two moments, affine components are already accounted for,
and the remaining discrepancy is governed by second-order smoothness.

The contributions of the paper are twofold. First, we develop a harmonised
and fully explicit version of Stein's method for the Gamma distribution on
\((0,\infty)\). The approach is based on kernel representations of the Stein
solution and of its first two derivatives. It is self-contained, keeps the
weights generated by the Gamma operator visible, and gives constants that can
be tracked explicitly. This yields weighted Stein factors adapted to the
distances \(d_1\) and \(d_2\).

Several uniform bounds for solutions of the Gamma--Stein equation and their
derivatives are already available in the literature, starting with the
generator approach developed in \cite{Luk1994,Pickett2004,GauntThesis} and
given in detail in \cite{Gaunt2017,Dobler2018}. The perspective taken here is
different. Rather than deriving only global derivative bounds, we give
pointwise kernel representations for the solution and for the derivatives
entering the Stein identity. These representations make the dependence on the
distribution parameters explicit and transparent, which is crucial for the
asymptotic analysis carried out later in the paper. In particular, they allow
us to obtain sharp large-shape asymptotics together with simple explicit
bounds valid in the corresponding asymptotic regimes.

Second, we apply this framework to Satterthwaite's approximation for sums of
independent Gamma random variables. Let
\[
  \Gw_N=\sum_{i=1}^N G_i,
  \qquad
  G_i\sim\mathcal G_{\alpha_i,\beta_i},
  \qquad
  i=1,\ldots,N,
\]
where \(G_1,\ldots,G_N\) are independent. Let
\(U_N\sim\mathcal G_{\alpha,\beta}\) be chosen by matching the mean and
variance of \(\Gw_N\), namely
\[
  \alpha\beta=\sum_{i=1}^N\alpha_i\beta_i,
  \qquad
  \alpha\beta^2=\sum_{i=1}^N\alpha_i\beta_i^2.
\]
Writing
\[
  A_N:=
  \max_{1\le i\le N}
  \left\{
    \left|\frac{\beta_i}{\beta}-1\right|\beta_i
  \right\},
  \qquad
  B_N:=
  \max_{1\le i\le N}
  \left\{
    \left|\frac{\beta_i}{\beta}-1\right|\beta_i^2
  \right\},
\]
we prove the non-asymptotic bounds
\[
  d_1(\Gw_N,U_N)
  \le
  2A_N
\]
and
\[
  d_2(\Gw_N,U_N)
  \le
  2B_N+\alpha\beta e_\alpha A_N .
\]
Thus the quality of the approximation is controlled directly by the
discrepancies between the individual scale parameters \(\beta_i\) and the
matched scale parameter \(\beta\). Combining the \(d_1\)-estimate with
\eqref{kolmosmoo}, {or its analogue of Proposition~\ref{prop:Kol-small-shape},} also gives a Kolmogorov bound. To the best of our knowledge, these are the
first Gamma--Stein bounds stated directly for the Satterthwaite
moment-matched approximation in terms of the parameters of the underlying
weighted Gamma sum.

We conclude the introduction with an overview of the paper. 
Section~\ref{sec:steingamma} develops the Gamma--Stein framework used
throughout the paper. We derive the Stein equation, its kernel solution,
the weighted Stein-factor estimates needed for the \(d_1\)- and
\(d_2\)-bounds{, and a Wasserstein-to-Kolmogorov inequality for Gamma target with shape parameter at most one}. Section~\ref{sec:main_result} applies these estimates to sums
of independent Gamma random variables and proves the main non-asymptotic
bounds. Section~\ref{sec:satt} specializes the abstract result to
Satterthwaite approximations in statistical examples, including pooled
variances, Gaussian time series, Wishart traces, sums of independent Wishart
traces, and high-dimensional repeated-measures designs. Finally, technical bounds for
the Gamma kernels, the explicit constants {and the Wasserstein-to-Kolmogorov bound} are collected in  Appendix~\ref{appendix:steinstuff}.

\section{Kernel-based Gamma--Stein's method on the positive real line}
\label{sec:steingamma}

In this section we develop the Gamma--Stein machinery used in the proof of
the main approximation result. We start from the classical Gamma
integration-by-parts identity, introduce the corresponding Stein equation,
and then derive kernel representations for the solution and its first two
derivatives. { All proofs are provided in Appendix \ref{appendix:steinstuff}.}

Throughout, \(a>0\) and \(b>0\) are fixed. We write
\(L^1(\mathcal G_{a,b})\) for the Borel functions
\(h:(0,\infty)\to\mathbb R\) such that
\(\mathbb E[|h(X)|]<\infty\) for \(X\sim\mathcal G_{a,b}\), and set
\[
  \mathcal G_{a,b}h:=\int_0^\infty h(u)p_{a,b}(u)\,du.
\]
For \(x>0\), write
\[
  P_{a,b}(x):=\int_0^x p_{a,b}(u)\,du,
  \qquad
  \overline P_{a,b}(x):=\int_x^\infty p_{a,b}(u)\,du.
\]

The basic Gamma--Stein machinery follows from the density approach; see, for
instance, \cite{dobler2015stein,ley2017distances,ley2017stein}. {  We begin with an integration-by-parts identity and the corresponding
characterization of the Gamma distribution}.

\begin{thm}
\label{thm:gamma-ibp}
Let \(X\sim\mathcal G_{a,b}\). If \(f:(0,\infty)\to\mathbb R\) is absolutely
continuous and \(x\mapsto xf'(x)\) belongs to \(L^1(\mathcal G_{a,b})\), then
\(x\mapsto (x-ab)f(x)\) belongs to \(L^1(\mathcal G_{a,b})\) and
\begin{equation}
\label{eq:gamma-ibp}
  \mathbb E[(X-ab)f(X)]
  =
  b\,\mathbb E[Xf'(X)] .
\end{equation}
Conversely, if a real-valued random variable \(X\) satisfies
\eqref{eq:gamma-ibp} for some \(a,b>0\) and every smooth compactly supported
function \(f:\mathbb R\to\mathbb R\), then \(X\sim\mathcal G_{a,b}\).
\end{thm}

{
\begin{rem}
Theorem~\ref{thm:gamma-ibp} is a restatement of the  classical Stein characterization of the gamma
distribution. The underlying integration-by-parts identity goes back to
\citet{diaconis_zabell}. The characterization follows from the generator
formulation of gamma approximation developed by \citet{Luk1994} and was
stated explicitly by \citet{Gaunt2017}. These formulations take the same form as \eqref{eq:gamma-ibp} but differ in their choice of test-function class. 
\end{rem}
}

The Stein equation associated with \eqref{eq:gamma-ibp} is obtained by
inverting the operator \(f\mapsto bx f'(x)-(x-ab)f(x)\) { (this differential operator is called {\it a} Gamma Stein operator)}. 

\begin{defn}
\label{def:steinequ}
Let \(h\in L^1(\mathcal G_{a,b})\). The Gamma--Stein equation associated with
\(h\) is
\begin{equation}
\label{eq:stein-equation-pos}
  bx f'(x)-(x-ab)f(x)
  =
  h(x)-\mathcal G_{a,b}h,
  \qquad x>0.
\end{equation}
A solution of \eqref{eq:stein-equation-pos} is a locally absolutely
continuous function \(f_{h;a,b}:(0,\infty)\to\mathbb R\) for which there
exists a version of its a.e.\ derivative \(f'_{h;a,b}\) such that
\eqref{eq:stein-equation-pos} holds for every \(x>0\).
\end{defn}

We then collect the kernels and kernel bounds needed in the sequel. For
\(y>0\), define the weighted lower and upper Gamma--Mills ratios by
\[
  \underline M_{a,b}(y):=\frac{P_{a,b}(y)}{byp_{a,b}(y)},
  \qquad
  \overline M_{a,b}(y):=\frac{\overline P_{a,b}(y)}{byp_{a,b}(y)}.
\]
We begin with an  elementary property which seems to be new; it follows from a simple integral
representation given in Appendix~\ref{appendix:steinstuff}.
Related convexity and reciprocal convexity properties of Gamma--Mills ratios
can be found in \cite{Baricz2012MillsReciprocalConvexity}.

\begin{lem}
\label{lem:gamma-mills-ratios}
The function \(y\mapsto \underline M_{a,b}(y)\) is absolutely monotone on
\((0,\infty)\), while \(y\mapsto \overline M_{a,b}(y)\) is completely
monotone on \((0,\infty)\). Equivalently, for every \(n\in\mathbb N\) and
\(y>0\),
\[
  \underline M_{a,b}^{(n)}(y)\ge0,
  \qquad
  (-1)^n\overline M_{a,b}^{(n)}(y)\ge0.
\]
In particular, \(\underline M_{a,b}\) is increasing and convex with
increasing second derivative, whereas \(\overline M_{a,b}\) is decreasing and
convex with decreasing second derivative.
\end{lem}

In particular the lower and upper Gamma--Mills ratios are therefore three times
continuously differentiable on \((0,\infty)\). For \(x,y>0\), define
\begin{align}
K_{a,b}^{0,1}(x,y)
&=
-\overline M_{a,b}(y)
 \frac{P_{a,b}(x)}{p_{a,b}(x)}
 \mathbf{1}_{\{x\le y\}}
-\underline M_{a,b}(y)
 \frac{\overline P_{a,b}(x)}{p_{a,b}(x)}
 \mathbf{1}_{\{x>y\}},
\label{eq:kernel-01}
\\
K_{a,b}^{1,1}(x,y)
&=
-\overline M_{a,b}'(y)
 \frac{P_{a,b}(x)}{p_{a,b}(x)}
 \mathbf{1}_{\{x\le y\}}
-\underline M_{a,b}'(y)
 \frac{\overline P_{a,b}(x)}{p_{a,b}(x)}
 \mathbf{1}_{\{x>y\}},
\label{eq:kernel-11}
\\
K_{a,b}^{2,1}(x,y)
&=
-\overline M_{a,b}''(y)
 \frac{P_{a,b}(x)}{p_{a,b}(x)}
 \mathbf{1}_{\{x\le y\}}
-\underline M_{a,b}''(y)
 \frac{\overline P_{a,b}(x)}{p_{a,b}(x)}
 \mathbf{1}_{\{x>y\}}.
\label{eq:kernel-21}
\end{align}
We also use the second-order kernels
\begin{align}
K_{a,b}^{2,2}(x,y)
&=
\overline M_{a,b}''(y)
 \frac{\underline{\mathrm{P}}^1_{a,b}(x)}{p_{a,b}(x)}
 \mathbf{1}_{\{x\le y\}}
-\underline M_{a,b}''(y)
 \frac{\overline{\mathrm{P}}^1_{a,b}(x)}{p_{a,b}(x)}
 \mathbf{1}_{\{x>y\}},
\label{eq:kernel-22}
\\
K_{a,b}^{3,2}(x,y)
&=
\overline M_{a,b}'''(y)
 \frac{\underline{\mathrm{P}}^1_{a,b}(x)}{p_{a,b}(x)}
 \mathbf{1}_{\{x\le y\}}
-\underline M_{a,b}'''(y)
 \frac{\overline{\mathrm{P}}^1_{a,b}(x)}{p_{a,b}(x)}
 \mathbf{1}_{\{x>y\}},
\label{eq:kernel-32}
\end{align}
where
\[
  \underline{\mathrm{P}}^1_{a,b}(x):=\int_0^x P_{a,b}(u)\,du,
  \qquad
  \overline{\mathrm{P}}^1_{a,b}(x):=\int_x^\infty \overline P_{a,b}(u)\,du.
\]
Finally, for \(y>0\), set
\begin{align}
C_{a,b}^{1}(y)
&:=
- \overline M_{a,b}'(y)\underline{\mathrm{P}}^1_{a,b}(y)
+
\underline M_{a,b}'(y)\overline{\mathrm{P}}^1_{a,b}(y),
\label{eq:M11-def}
\\
C_{a,b}^{2}(y)
&:=
\overline M_{a,b}''(y)\int_0^y \underline{\mathrm{P}}^1_{a,b}(x)\,dx
+
\underline M_{a,b}''(y)\int_y^\infty \overline{\mathrm{P}}^1_{a,b}(x)\,dx.
\label{eq:M22-def}
\end{align}
By scaling,
\begin{equation}
\label{eq:M-scaling}
  C_{a,b}^{1}(y)
  =
  \frac1b C_{a,1}^{1}\left(\frac{y}{b}\right),
  \qquad
  C_{a,b}^{2}(y)
  =
  \frac1b C_{a,1}^{2}\left(\frac{y}{b}\right).
\end{equation}

The basic kernel estimates are as follows.

\begin{lem}
\label{lma:boundskern}
Let \(X\sim\mathcal G_{a,b}\). Then, for every \(y>0\),
\[
  \mathbb E[|K_{a,b}^{0,1}(X,y)|]=1,
  \qquad
  \mathbb E[|K_{a,b}^{2,1}(X,y)|]
  =
  \mathbb E[|K_{a,b}^{3,2}(X,y)|]
  =
  \frac1{by},
\]
and
\[
  \mathbb E[|K_{a,b}^{1,1}(X,y)|]=C_{a,b}^{1}(y),
  \qquad
  \mathbb E[|K_{a,b}^{2,2}(X,y)|]=C_{a,b}^{2}(y).
\]
\end{lem}

We now connect the kernels with solutions of the Gamma--Stein equation
\eqref{eq:stein-equation-pos}. Let \(X\sim\mathcal G_{a,b}\). For \(x>0\),
define
\begin{align}
f_{h;a,b}(x)
&=
\mathbb E[h'(X)K_{a,b}^{0,1}(X,x)],
\label{eq:fh-explicit}
\\
g_{h;a,b}(x)
&=
\mathbb E[h'(X)K_{a,b}^{1,1}(X,x)],
\label{eq:gh-explicit}
\\
\ell_{h;a,b}(x)
&=
\mathbb E[h'(X)K_{a,b}^{2,1}(X,x)]
+\frac{h'(x)}{bx},
\label{eq:ellh-explicit}
\\
\widetilde\ell_{h;a,b}(x)
&=
\mathbb E[h''(X)K_{a,b}^{2,2}(X,x)],
\label{eq:tildelh-explicit}
\\
m_{h;a,b}(x)
&=
\mathbb E[h''(X)K_{a,b}^{3,2}(X,x)]
+\frac{h''(x)}{bx}.
\label{eq:mh-explicit}
\end{align}
These representations require the corresponding regularity of \(h\):
\eqref{eq:fh-explicit} and \eqref{eq:gh-explicit} require \(h\) to be a.e.
differentiable; \eqref{eq:ellh-explicit} requires differentiability on
\((0,\infty)\); \eqref{eq:tildelh-explicit} requires \(h\) to be twice a.e.
differentiable; and \eqref{eq:mh-explicit} requires twice differentiability
on \((0,\infty)\).

\begin{lem}
\label{lem:solution}
Fix \(a>0\) and \(b>0\). If \(h\in\mathcal Z_1\),  the function
\(f_{h;a,b}\) defined in \eqref{eq:fh-explicit} solves
\eqref{eq:stein-equation-pos}. Moreover, \(f_{h;a,b}\) and \(g_{h;a,b}\) are
differentiable on \((0,\infty)\), with
\(f_{h;a,b}'=g_{h;a,b}\) and \(g_{h;a,b}'=\ell_{h;a,b}\). The functions
\(f_{h;a,b}\), \(g_{h;a,b}\), \(xg_{h;a,b}\), and \(x\ell_{h;a,b}\) are
bounded on \((0,\infty)\), hence integrable with respect to
\(\mathcal G_{u,v}\) for every \(u,v>0\).

{ If \(h\in\mathcal Z_2\), then \(f_{h;a,b}\) also solves
\eqref{eq:stein-equation-pos}, with \(f_{h;a,b}'=g_{h;a,b}\) and
\(g_{h;a,b}'=\ell_{h;a,b}=\widetilde\ell_{h;a,b}\) on \((0,\infty)\).}
Moreover,
\(\ell_{h;a,b}\) is differentiable on \((0,\infty)\), with
\(\ell_{h;a,b}'=m_{h;a,b}\). The functions \(\ell_{h;a,b}\) and
\(xm_{h;a,b}\) are bounded on \((0,\infty)\), hence integrable with respect
to \(\mathcal G_{u,v}\) for every \(u,v>0\).
\end{lem}

Lemma~\ref{lem:solution} gives pointwise representatives of the Stein
solution and of its derivatives. From now on, for convenience and under the
regularity assumptions on \(h\) required by Lemma~\ref{lem:solution}, we use
the notation
\[
  f_h:=f_{h;a,b},
  \qquad
  f_h':=g_{h;a,b},
  \qquad
  f_h'':=\ell_{h;a,b},
  \qquad
  f_h''':=m_{h;a,b},
\]
where the derivatives are understood in the pointwise representative sense
given by Lemma~\ref{lem:solution}. In particular, these representatives satisfy
the weighted boundedness properties needed to apply the Gamma
integration-by-parts identity. The following  bounds then follow directly from
Lemma~\ref{lma:boundskern} after taking absolute values.

\begin{lem}[Non-uniform Stein factors]
\label{lem:stein-factors-nonuniform}
Let \(a>0\) and \(b>0\), and put \(z=x/b\). If \(h\in\mathcal Z_1\), then,
for every \(x>0\),
\[
  |f_h(x)|\le1,
  \qquad
  |f_h'(x)|\le\frac{1}{b}C_{a,1}^{1}(z),
  \qquad
  |x f_h'(x)|\le zC_{a,1}^{1}(z),
  \qquad
  |x f_h''(x)|\le\frac2b.
\]
If \(h\in\mathcal Z_2\), then, for every \(x>0\),
\[
  |f_h''(x)|\le\frac{1}{b}C_{a,1}^{2}(z),
  \qquad
  |x f_h'''(x)|\le\frac2b.
\]
\end{lem}

\begin{rem}[Sharpness of the non-uniform Stein factors]
\label{rem:sharpness-nonuniform-stein-factors}
The non-uniform bounds in Lemma~\ref{lem:stein-factors-nonuniform} cannot be
improved as bounds uniform over \(h\in\mathcal Z_1\) or \(h\in\mathcal Z_2\).
Indeed, each estimate is obtained from a kernel representation by taking the
absolute value inside the expectation. This loss is unavoidable.

More precisely, the sign properties of the Gamma--Mills' ratios from
Lemma~\ref{lem:gamma-mills-ratios} determine the signs of the relevant
kernels: \(K_{a,b}^{0,1}\), \(K_{a,b}^{2,1}\), and \(K_{a,b}^{3,2}\) have a
fixed sign, whereas \(K_{a,b}^{1,1}\) and \(K_{a,b}^{2,2}\) change sign only
at the evaluation point. Therefore the usual \(L^\infty\)--\(L^1\) duality
identity
\[
  \sup_{\|\varphi\|_\infty\le1}
  \left|
    \mathbb E[\varphi(X)K(X,x)]
  \right|
  =
  \mathbb E[|K(X,x)|]
\]
shows that the envelopes appearing in Lemma~\ref{lem:stein-factors-nonuniform}
are the best possible ones. The supremum is attained over bounded measurable functions by taking
\(\varphi\) to be the sign of the relevant kernel; within the original smooth
test classes, the same conclusion follows by smooth approximation.
\end{rem}

The pointwise optimal bounds
\(C_{a,1}^{1}(z)\), \(zC_{a,1}^{1}(z)\), and \(C_{a,1}^{2}(z)\)
are given explicitly in terms of the Gamma density and distribution function.
Their analytic properties can therefore be studied directly. In
Proposition~\ref{prop:Ca-bounds}, we obtain their behaviour at zero and at
infinity, and show in particular that they are bounded on \((0,\infty)\).
Thus
\[
  c_a:=\sup_{z>0}C_{a,1}^{1}(z),
  \qquad
  d_a:=\sup_{z>0}zC_{a,1}^{1}(z),
  \qquad
  e_a:=\sup_{z>0}C_{a,1}^{2}(z)
\]
are finite. Taking the supremum over the evaluation point \(z\) in
Lemma~\ref{lem:stein-factors-nonuniform} yields the following bounds, now
uniform in \(z\), for the corresponding Stein factors. 

\begin{lem}[Stein factors]
\label{lem:stein-factors}
Let \(a>0\) and \(b>0\). If \(h\in\mathcal Z_1\), then
\[
  \|f_h\|_\infty\le 1,
  \qquad
  \|f_h'\|_\infty\le\frac{c_a}{b},
  \qquad
  \|xf_h'\|_\infty\le d_a,
  \qquad
  \|xf_h''\|_\infty\le\frac2b.
\]
If \(h\in\mathcal Z_2\), then
\[
  \|f_h''\|_\infty\le\frac{e_a}{b},
  \qquad
  \|xf_h'''\|_\infty\le\frac2b.
\]
\end{lem}

Finally, Proposition~\ref{prop:Ca-bounds} gives, as \(a\to\infty\),
\begin{equation}
  \sqrt a\,c_a\to\sqrt{\frac{2}{\pi}},
  \qquad
  \frac{d_a}{\sqrt a}\to\sqrt{\frac{2}{\pi}},
  \qquad
  \sqrt a\,e_a\to\sqrt{\frac{\pi}{8}}.
\label{eq:asumoacde}
\end{equation}
These asymptotics, illustrated in
Figure~\ref{fig:gamma-stein-asymptotics}, make explicit the dependence of the
Stein factors on the shape parameter in the large-\(a\) regime, and will be
used below to compare our sharp estimates with existing bounds in the
literature.

\begin{rem}
\label{rem:explicit-bound-ea}
For completeness, we also mention that one can derive simple elementary
upper bounds for the constants appearing above. For instance, as shown in
Appendix~\ref{appendix:steinstuff}, for every {\(a\ge17\)},
\[
  e_a\le \frac{5}{2\sqrt a}.
\]
The constant \(5/2\) is not sharp and is chosen only for convenience. This
bound is not needed in its explicit form below, where we keep the notation
\(e_a\), but it confirms that \(e_a\) has the correct order \(a^{-1/2}\), in
agreement with \eqref{eq:asumoacde}.
\end{rem}
\begin{figure}[t]
\centering
\begin{subfigure}{0.32\textwidth}
  \centering
  \includegraphics[width=\textwidth]{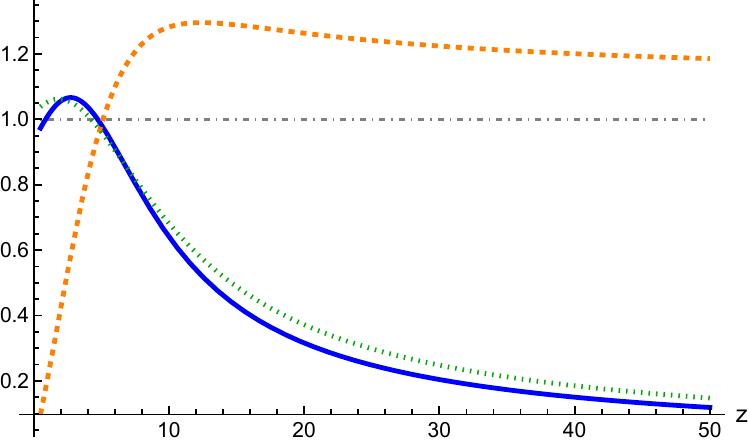}
  \caption{\(a=5\)}
\end{subfigure}\hfill
\begin{subfigure}{0.32\textwidth}
  \centering
  \includegraphics[width=\textwidth]{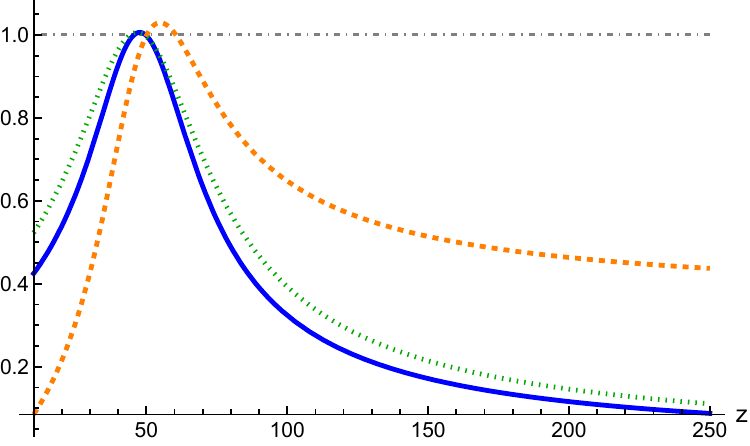}
  \caption{\(a=50\)}
\end{subfigure}\hfill
\begin{subfigure}{0.32\textwidth}
  \centering
  \includegraphics[width=\textwidth]{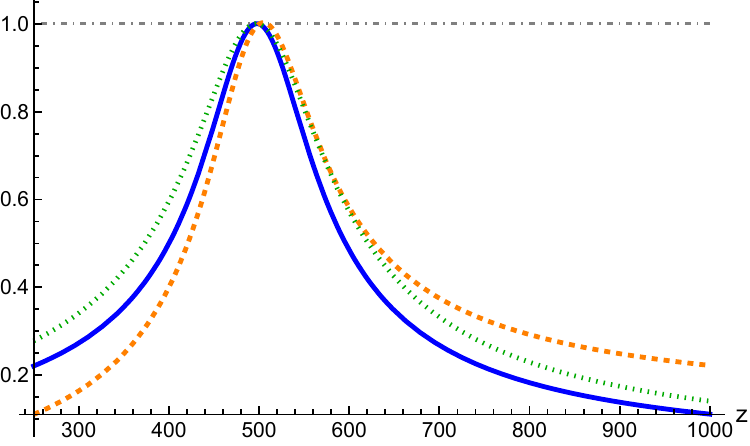}
  \caption{\(a=500\)}
\end{subfigure}
\caption{{\it Normalized non-uniform Stein-factor envelopes for increasing values
of the Gamma shape parameter. From left to right, the plots correspond to
\(a=5\), \(a=50\), and \(a=500\). In each panel, the solid blue curve is
\(C_{a,1}^{1}(z)/(\sqrt{2/(\pi a)})\), the dashed orange curve is
\(zC_{a,1}^{1}(z)/(\sqrt{2a/\pi})\), and the dotted green curve is
\(C_{a,1}^{2}(z)/\sqrt{\pi/(8a)}\). The grey dot-dashed horizontal line marks
the reference level \(1\). The stabilization of the curves illustrates the
large-shape asymptotics in \eqref{eq:asumoacde}.}}
\label{fig:gamma-stein-asymptotics}
\end{figure}
\begin{rem}
\label{rem:comparison-gpr}
The closest comparable bounds we are aware of are those of
\cite{Gaunt2017,GauntThesis}. Translated into our notation, they imply that,
for every \(m\ge1\),
\[
  \|f_{h;a,b}^{(m)}\|_\infty
  \le
  \frac1b
  \left\{
  \sqrt{\frac{2\pi}{a+m}}
  +
  \frac{\log(a+m+1)}{e(a+m)}
  +
  \frac{2}{a+m}
  +
  \frac{1}{e(a+m+1)}
  \right\}
  \|h^{(m)}\|_\infty .
\]
For \(m=1\) and \(m=2\), these bounds have the same order as the estimates
obtained above, namely \(a^{-1/2}/b\), but with different leading constants:
our leading constants are smaller by asymptotic factors \(\pi\) for the first
derivative and \(4\) for the second derivative. Since we do not pursue
higher-order derivatives here, we make no formal comparison for \(m\ge3\).

A similar comparison can be made for weighted bounds. Lemma~2.1 in
\cite{Gaunt2017} translates, in our notation, into
\[
  \|x f_{h;a,b}^{(m+1)}\|_\infty
  \le
  \frac4b\|h^{(m)}\|_\infty,
  \qquad m\ge1.
\]
For the derivatives considered here, our corresponding weighted estimates
have constant \(2/b\), and hence improve this bound by a factor \(2\). This
constant is the sharp one for the kernel representations used above, in the
sense explained in Remark~\ref{rem:sharpness-nonuniform-stein-factors}.

Finally, the comparison with \cite{Dobler2018} is less direct, since they study 
the Gamma Stein equation on the whole real line rather than only
on the support \((0,\infty)\) of the target distribution. The kernel
representations we propose are intrinsically tied to the positive half-line,
so it is not immediate whether they admit a natural analogue in that setting.
\end{rem}

We wrap up the section with the final ingredient of the
Gamma--Stein method: the transfer principle, which rewrites the
\(d_j\)-discrepancy from a Gamma target as an expected Stein operator.

\begin{cor}[Stein transfer principle]
\label{cor:sttranp}
Let \(Y\) be a positive random variable. Fix \(a>0\), \(b>0\), and let
\(G_{a,b}\sim\mathcal G_{a,b}\). Assume that \(Y\) has a finite first moment
if \(j=1\), and a finite second moment if \(j=2\). Then
\begin{equation}
\label{eq:stein-transfer1}
  d_j(Y,G_{a,b})
  =
  \sup_{h\in\mathcal Z_j}
  \left|
    \mathbb E\left[bYf_{h;a,b}'(Y)\right]
    -
    \mathbb E\left[(Y-ab)f_{h;a,b}(Y)\right]
  \right|,
\end{equation}
where \(f_{h;a,b}\) is the Stein solution defined in
\eqref{eq:fh-explicit}, and \(f_{h;a,b}'\) denotes the pointwise
representative \(g_{h;a,b}\) from \eqref{eq:gh-explicit}.
\end{cor}

The regularity and boundedness guaranteed by
Lemma~\ref{lem:solution}, together with the Stein factors from
Lemma~\ref{lem:stein-factors}, make \eqref{eq:stein-transfer1}
tractable in Gamma comparison problems.

{ To obtain a bound on the Kolmogorov distance, we will use inequality~\eqref{kolmosmoo} when $a \ge 1$, since in this case the density is bounded by 
\[B_{a,b}:=\sup_{x>0}p_{a,b}(x)\le\frac1b\min\Big\{1,\frac{1}{\sqrt{2\pi(a-1)}}\Big\},\]
 (with the convention $1/\sqrt{0}:=\infty$). When $a <1$, the density $p_{a,b}$ is unbounded near the origin and, as anticipated in the Introduction,  we shall use the following inequality instead.
\begin{prop}\label{prop:Kol-small-shape}
 Let $Y$ be a positive random variable. Fix $0<a\le 1$, $b>0$ and $G_{a,b} \sim \mathcal{G}_{a,b}$. Assume that $Y$ has a finite first moment. Then 
 \begin{equation}\label{eq:Kol-small-shape}
d_{\mathrm{Kol}}(Y,G_{a,b})\le K_a\Big(\frac{d_1(Y,G_{a,b})}{b}\Big)^{\frac{a}{a+1}},
\qquad
K_a:=\Gamma(a+1)^{-\frac{1}{a+1}}\Big(\frac{a+1}{a}\Big)^{\frac{a}{a+1}}.
\end{equation}    
\end{prop}

\begin{rem}\label{rem:Kol-small-shape}
Proposition~\ref{prop:Kol-small-shape} is a sharpened version, for Gamma targets, of the case $m=1$ of \citet[Prop.~2.1(iii) and 2.3(iii)]{Gaunt2023}; its proof follows closely the smoothing argument of \citet[proof of Corollary~1]{Gaunt2022}, exploiting the monotonicity of $p_{a,b}$ for $a\le1$. For $a=1$ it reduces to~\eqref{kolmosmoo} with $B_{1,b}=1/b$. The exponent $a/(a+1)$ cannot be improved in general: for $Z\sim \mathcal G_{a,1}$, $w\in(0,1)$ and $Y_w=\max(Z,w)$, one has $d_{\mathrm{Kol}}(Y_w,Z)\ge w^a/(e\,\Gamma(a+1))$ while $d_1(Y_w,Z)\le w^{a+1}/\Gamma(a+2)$.
\end{rem}
}

\section{Gamma approximation for sums of independent Gamma variates}
\label{sec:main_result}

In this section we turn the Gamma--Stein machinery of
Section~\ref{sec:steingamma} into quantitative bounds for sums of independent
Gamma random variables. Let \(N\ge1\), let
\(G_i\sim\mathcal G_{\alpha_i,\beta_i}\), \(i=1,\ldots,N\), be independent,
and set
\begin{equation}
\label{eq:def-GwN}
  \Gw_N:=\sum_{i=1}^N G_i .
\end{equation}
Let \(U_N\sim\mathcal G_{\alpha,\beta}\), where \(\alpha\) and \(\beta\) are
chosen by matching the mean and variance of \(\Gw_N\):
\begin{equation}
\label{eq:moment-matching}
  \alpha\beta=\sum_{i=1}^N\alpha_i\beta_i,
  \qquad
  \alpha\beta^2=\sum_{i=1}^N\alpha_i\beta_i^2 .
\end{equation}
Equivalently,
\[
  \beta=
  \frac{\sum_{i=1}^N\alpha_i\beta_i^2}
       {\sum_{i=1}^N\alpha_i\beta_i},
  \qquad
  \alpha=
  \frac{\left(\sum_{i=1}^N\alpha_i\beta_i\right)^2}
       {\sum_{i=1}^N\alpha_i\beta_i^2}.
\]
With this choice, \(\Gw_N\) and \(U_N\) have the same mean and variance. The
approximation error is controlled by the discrepancies between the individual
scales \(\beta_i\) and the matched scale \(\beta\).

\begin{thm}
\label{thm:wass_gamma_stein-gamma}
Assume that \(G_i\sim\mathcal G_{\alpha_i,\beta_i}\),
\(i=1,\ldots,N\), are independent, and let
\(U_N\sim\mathcal G_{\alpha,\beta}\), where \(\alpha,\beta\) are chosen by
\eqref{eq:moment-matching}. Set
\begin{equation}
\label{eq:ANBN}  A_N:=
  \max_{1\le i\le N}
  \left\{
    \left|\frac{\beta_i}{\beta}-1\right|\beta_i
  \right\},
  \qquad
  B_N:=
  \max_{1\le i\le N}
  \left\{
    \left|\frac{\beta_i}{\beta}-1\right|\beta_i^2
  \right\}.
  \end{equation}
 Then
\[
  d_1(\Gw_N,U_N)\le 2A_N,
  \qquad
  d_2(\Gw_N,U_N)
  \le
  2B_N+\alpha\beta e_\alpha A_N.
\]
Moreover, if \(\alpha\ge 1\), then
\begin{equation}
\label{eq:kolmogorov-bound}
  d_{\mathrm{Kol}}(\Gw_N,U_N)
  \le {2\Big(\frac{A_N}{\beta}\min\Big\{1,\frac{1}{\sqrt{2\pi(\alpha-1)}}\Big\}\Big)^{1/2}.}
\end{equation}
{ and, if $0<\alpha<1$, then
\begin{equation}\label{eq:Kol-small-shape-GN}
d_{\mathrm{Kol}}(\Gw_N,U_N)\le K_\alpha\Big(\frac{2A_N}{\beta}\Big)^{\frac{\alpha}{1+\alpha}},
\end{equation}
where $K_\alpha$ is defined in Proposition~\ref{prop:Kol-small-shape} .}
\end{thm}

\begin{proof}
Let \(h\in\mathcal Z_1\). Write
\[
  f_h=f_{h;\alpha,\beta},
  \qquad
  g_h=g_{h;\alpha,\beta},
  \qquad
  \ell_h=\ell_{h;\alpha,\beta}.
\]
By Lemma~\ref{lem:solution}, \(f_h'=g_h\) and \(g_h'=\ell_h\) on
\((0,\infty)\). The Stein transfer principle gives
\begin{equation}
\label{eq:stein-start-adaptive-proof}
  \E[h(\Gw_N)]-\E[h(U_N)]
  =
  \E[\beta \Gw_N g_h(\Gw_N)]
  -
  \E[(\Gw_N-\alpha\beta)f_h(\Gw_N)] .
\end{equation}

For each \(i\), set \(\Gw_N^{(i)}:=\Gw_N-G_i\). Since
\(\sum_i\alpha_i\beta_i=\alpha\beta\),
\[
  \E[(\Gw_N-\alpha\beta)f_h(\Gw_N)]
  =
  \sum_{i=1}^N
  \E[(G_i-\alpha_i\beta_i)f_h(G_i+\Gw_N^{(i)})].
\]
For fixed \(i\), put \(\psi_i(x):=\E[f_h(x+\Gw_N^{(i)})]\). Since \(G_i\)
and \(\Gw_N^{(i)}\) are independent, and since
\(\psi_i'(x)=\E[g_h(x+\Gw_N^{(i)})]\), the Gamma integration-by-parts identity
gives
\[
  \E[(\Gw_N-\alpha\beta)f_h(\Gw_N)]
  =
  \sum_{i=1}^N
  \beta_i\E[G_i g_h(\Gw_N)].
\]
The required integrability follows from Lemma~\ref{lem:solution}, since
\[
  \E[G_i|\psi_i'(G_i)|]
  \le
  \E\left[\frac{G_i}{\Gw_N}\,\Gw_N|g_h(\Gw_N)|\right]
  <\infty .
\]
Substituting this identity into \eqref{eq:stein-start-adaptive-proof} yields
\[
  \E[h(\Gw_N)]-\E[h(U_N)]
  =
  -\beta
  \sum_{i=1}^N
  \left(\frac{\beta_i}{\beta}-1\right)
  \E[G_i g_h(\Gw_N)] .
\]
By variance matching,
\[
  \sum_{i=1}^N
  \left(\frac{\beta_i}{\beta}-1\right)\alpha_i\beta_i
  =
  \frac1\beta\sum_{i=1}^N\alpha_i\beta_i^2
  -
  \sum_{i=1}^N\alpha_i\beta_i
  =
  0.
\]
Hence
\[
  \E[h(\Gw_N)]-\E[h(U_N)]
  =
  -\beta
  \sum_{i=1}^N
  \left(\frac{\beta_i}{\beta}-1\right)
  \E[(G_i-\alpha_i\beta_i)g_h(\Gw_N)] .
\]
Applying integration by parts once more, now to
\(x\mapsto\E[g_h(x+\Gw_N^{(i)})]\), gives
\begin{equation}
\label{eq:maineq}
  \E[h(\Gw_N)]-\E[h(U_N)]
  =
  -\beta
  \sum_{i=1}^N
  \left(\frac{\beta_i}{\beta}-1\right)
  \beta_i
  \E[G_i\ell_h(\Gw_N)] .
\end{equation}
Using \(\|x\ell_h\|_\infty\le2/\beta\), we obtain
\[
  |\E[h(\Gw_N)]-\E[h(U_N)]|
  \le
  2A_N
  \sum_{i=1}^N
  \E\left[\frac{G_i}{\Gw_N}\right]
  =
  2A_N.
\]
Taking the supremum over \(h\in\mathcal Z_1\) proves the \(d_1\)-bound.

We now let \(h\in\mathcal Z_2\). Starting from \eqref{eq:maineq}, write
\[
\begin{aligned}
  \E[h(\Gw_N)]-\E[h(U_N)]
  &=
  -\beta
  \sum_{i=1}^N
  \left(\frac{\beta_i}{\beta}-1\right)
  \beta_i
  \E[(G_i-\alpha_i\beta_i)\ell_h(\Gw_N)] \\
  &\quad
  -\beta
  \sum_{i=1}^N
  \left(\frac{\beta_i}{\beta}-1\right)
  \alpha_i\beta_i^2
  \E[\ell_h(\Gw_N)] .
\end{aligned}
\]
Applying integration by parts to
\(x\mapsto\E[\ell_h(x+\Gw_N^{(i)})]\), and using \(\ell_h'=m_h\), gives
\[
\begin{aligned}
  \E[h(\Gw_N)]-\E[h(U_N)]
  &=
  -\beta
  \sum_{i=1}^N
  \left(\frac{\beta_i}{\beta}-1\right)
  \beta_i^2
  \E[G_i m_h(\Gw_N)]  
  -\beta
  \sum_{i=1}^N
  \left(\frac{\beta_i}{\beta}-1\right)
  \alpha_i\beta_i^2
  \E[\ell_h(\Gw_N)] .
\end{aligned}
\]
The Stein-factor bounds \(\|xm_h\|_\infty\le2/\beta\) and
\(\|\ell_h\|_\infty\le e_\alpha/\beta\) imply
\[
\begin{aligned}
  |\E[h(\Gw_N)]-\E[h(U_N)]|
  &\le
  2
  \sum_{i=1}^N
  \left|\frac{\beta_i}{\beta}-1\right|\beta_i^2
  \E\left[\frac{G_i}{\Gw_N}\right]  
  +e_\alpha
  \sum_{i=1}^N
  \left|\frac{\beta_i}{\beta}-1\right|\alpha_i\beta_i^2 .
\end{aligned}
\]
The first term is at most \(2B_N\), because
\(\sum_i\E[G_i/\Gw_N]=1\). For the second term,
\[
  \sum_{i=1}^N
  \left|\frac{\beta_i}{\beta}-1\right|\alpha_i\beta_i^2
  \le
  A_N\sum_{i=1}^N\alpha_i\beta_i
  =
  A_N\alpha\beta .
\]
Thus
\[
  |\E[h(\Gw_N)]-\E[h(U_N)]|
  \le
  2B_N+\alpha\beta e_\alpha A_N.
\]
Taking the supremum over \(h\in\mathcal Z_2\) proves the \(d_2\)-bound.

{It remains to prove the Kolmogorov bounds. If $\alpha\ge1$, the smoothing inequality~\eqref{kolmosmoo} and the $d_1$-bound give
\[
d_{\mathrm{Kol}}(\Gw_N,U_N)\le\sqrt{2B_{\alpha,\beta}\,d_1(\Gw_N,U_N)}
\le 2\Big(\frac{A_N}{\beta}\min\Big\{1,\frac{1}{\sqrt{2\pi(\alpha-1)}}\Big\}\Big)^{1/2}.
\]
If $0<\alpha<1$, \eqref{eq:Kol-small-shape-GN} follows from Proposition~\ref{prop:Kol-small-shape} with $(a,b)=(\alpha,\beta)$, the $d_1$-bound $d_1(\Gw_N,U_N)\le2A_N$, and the monotonicity of $x\mapsto x^{\alpha/(1+\alpha)}$.}

\end{proof}

Theorem~\ref{thm:wass_gamma_stein-gamma} reduces the approximation problem
to { the study of the quantities  \(A_N\) and \(B_N\) defined in \eqref{eq:ANBN}}. Under \eqref{eq:moment-matching}, the matched scale
\(\beta\) is a weighted mean of the individual scales:
\begin{equation}
\label{eq:beta_wmean}
  \beta
  =
  \sum_{i=1}^N w_i\beta_i,
  \qquad
  w_i
  :=
  \frac{\alpha_i\beta_i}
       {\sum_{j=1}^N\alpha_j\beta_j}.
\end{equation}
Here \(w_i>0\) and \(\sum_iw_i=1\). Thus, whenever the individual scales are
all of the same order, the matched scale is of that order as well.
{
\begin{lem}\label{lem:condition_rate_discrepancy}
Let $N\ge1$, let $\alpha_i,\beta_i>0$ for $1\le i\le N$, and let
$\alpha,\beta$ be defined by~\eqref{eq:moment-matching}. Set
\[
  \varepsilon_N:=\max_{1\le i\le N}
  \left|\frac{\beta_i}{\beta}-1\right|,
  \qquad
  \kappa_N:=\max_{1\le i\le N}\frac{\beta_i}{\beta}.
\]
Then $\min_i\beta_i\le\beta\le\max_i\beta_i$, so that $\kappa_N\ge1$, and
\[
  A_N\le\beta\kappa_N\varepsilon_N,
  \qquad
  B_N\le\beta^2\kappa_N^2\varepsilon_N.
\]
Consequently,
\begin{align*}
  d_1(\Gw_N,U_N)
  &\le 2\beta\kappa_N\varepsilon_N,\\
  d_2(\Gw_N,U_N)
  &\le \beta^2\kappa_N\varepsilon_N
       (2\kappa_N+\alpha e_\alpha),\\
  d_{\mathrm{Kol}}(\Gw_N,U_N)
  &\le
  \begin{cases}
    2(\kappa_N\varepsilon_N)^{1/2}
    \min\big\{1,(2\pi(\alpha-1))^{-1/4}\big\},
      & \alpha>1,\\[1ex]
    K_\alpha(2\kappa_N\varepsilon_N)^{\alpha/(1+\alpha)},
      & 0<\alpha\le1.
  \end{cases}
\end{align*}
In particular, suppose that $\beta_i\asymp\tau$ uniformly in $i$, for
some $\tau>0$ which may itself depend on the parameters. Then
$\beta\asymp\tau$, $\kappa_N=O(1)$, and
\begin{align*}
  d_1(\Gw_N,U_N)
  &=O(\beta\varepsilon_N),\\
  d_2(\Gw_N,U_N)
  &=O\big(\beta^2\varepsilon_N(1+\alpha e_\alpha)\big),\\
  d_{\mathrm{Kol}}(\Gw_N,U_N)
  &=
  \begin{cases}
    O\big(\varepsilon_N^{1/2}
    \min\{1,(\alpha-1)^{-1/4}\}\big),
      & \alpha>1,\\[1ex]
    O\big(\varepsilon_N^{\alpha/(1+\alpha)}\big),
      & 0<\alpha\le1.
  \end{cases}
\end{align*}
The implicit constant in the last bound is uniform in $\alpha\in(0,1]$.
\end{lem}

\begin{proof}
By construction,  $\beta$ is a weighted mean of $\beta_1,\dots,\beta_N$, whence
$\min_i\beta_i\le\beta\le\max_i\beta_i$ and $\kappa_N\ge1$. Moreover,
\begin{align*}
  A_N
  &=\beta\max_i\left\{
    \left|\frac{\beta_i}{\beta}-1\right|
    \frac{\beta_i}{\beta}\right\}
    \le\beta\kappa_N\varepsilon_N, \quad 
  B_N
  =\beta^2\max_i\left\{
    \left|\frac{\beta_i}{\beta}-1\right|
    \left(\frac{\beta_i}{\beta}\right)^2\right\}
    \le\beta^2\kappa_N^2\varepsilon_N.
\end{align*}
Inserting these estimates into Theorem~\ref{thm:wass_gamma_stein-gamma} gives the stated bounds.

If $\beta_i\asymp\tau$ uniformly in $i$, then the same comparison holds
for their weighted mean $\beta$, and hence $\kappa_N=O(1)$. The
$O$-bounds follow. For the last one, observe that $K_\alpha\le 8/5$ and
$(2\kappa_N)^{\alpha/(1+\alpha)}\le(2\kappa_N)^{1/2}$ throughout
$0<\alpha\le1$.
\end{proof}
}

{
\begin{rem}\label{rem:regimes}
Lemma~\ref{lem:condition_rate_discrepancy} makes explicit how the bounds
depend on the scale $\beta$, the spread of the $\beta_i$ around $\beta$,
and the shape $\alpha$. The scale enters through a factor $\beta$ in the
$d_1$ bound and $\beta^2$ in the $d_2$ bound. The Kolmogorov distance is
invariant under a common rescaling. All three bounds vanish when
$\varepsilon_N=0$. In that case all the $\beta_i$ are equal and $\Gw_N$
has exactly the law of $U_N$.

To see the effect of the shape parameter, take
$\sigma=\beta\sqrt{\alpha}$, the standard deviation of $U_N$. Dividing
$\Gw_N$ and $U_N$ by $\sigma$ replaces $\beta$ by $1/\sqrt{\alpha}$ and
preserves the ratios $\beta_i/\beta$. Applying
Lemma~\ref{lem:condition_rate_discrepancy} at this scale and using
Proposition~\ref{prop:ea-explicit} gives, for $\alpha\ge17$,
\begin{align*}
  d_1\left(\frac{\Gw_N}{\sigma},\frac{U_N}{\sigma}\right)
  &\le \frac{2\kappa_N\varepsilon_N}{\sqrt{\alpha}},\\
  d_2\left(\frac{\Gw_N}{\sigma},\frac{U_N}{\sigma}\right)
  &\le \frac{\kappa_N\varepsilon_N(2\kappa_N+5/2)}{\sqrt{\alpha}}.
\end{align*}
Thus, at the scale of the standard deviation, both bounds gain a factor
$\alpha^{-1/2}$. For the Kolmogorov distance, the gain comes from the
height of the target density. As $\alpha\to\infty$,
$\sup_x p_{\alpha,\beta}(x)\asymp(\beta\sqrt{\alpha})^{-1}$, and
the square root in the smoothing inequality gives a factor
$\alpha^{-1/4}$.

For $0<\alpha\le1$, the Kolmogorov bound is small when the scales
$\beta_i$ are close to one another. Its dependence on $\varepsilon_N$
weakens as $\alpha\downarrow0$, since the exponent $\alpha/(1+\alpha)$
tends to zero. This reflects the concentration of the rescaled target
$U_N/\beta$ near the origin.
\end{rem}
}

\section{Satterthwaite approximations}
\label{sec:satt}

We now return to the statistical motivation of the paper. The examples below
come from different settings---heteroscedastic variance estimation, dependent
Gaussian observations, Wishart traces, and high-dimensional repeated-measures
testing---but they share the same algebraic core. In each case, the statistic
can be diagonalized as a finite sum
\[
  W=\sum_{i=1}^N Y_i
\]
of independent Gamma random variables with possibly unequal scale parameters.
The Satterthwaite approximation replaces this sum by a single Gamma
distribution with matching mean and variance.

Throughout this section, we write
\[
  Y_i\sim\mathcal G_{\alpha_i,2\tau_i},
  \qquad i=1,\ldots,N,
\]
with \(Y_1,\ldots,Y_N\) independent. Thus the scale parameters in
Theorem~\ref{thm:wass_gamma_stein-gamma} are \(\beta_i=2\tau_i\). If the
moment-matched Gamma approximation is \(U\sim\mathcal G_{\nu/2,2\tau}\),
with \(\mu=\E[W]\) and \(\tau=\mu/\nu\), define
\[
  \Delta^{(r)}
  :=
  \max_i
  \left\{
  \left|\frac{\tau_i}{\tau}-1\right|\tau_i^r
  \right\},
  \qquad r=1,2.
\]
Then Theorem~\ref{thm:wass_gamma_stein-gamma} gives directly
\[
  d_1(W,U)\le 4\Delta^{(1)},\qquad
  d_2(W,U)\le 8\Delta^{(2)}+2\mu e_{\nu/2}\Delta^{(1)},
\]
{
and, by~\eqref{eq:kolmogorov-bound} and~\eqref{eq:Kol-small-shape-GN},
\begin{equation}\label{eq:Kol-Satterthwaite}
d_{\mathrm{Kol}}(W,U)\le
\begin{cases}
2\Big(\dfrac{\nu\Delta^{(1)}}{\mu}\min\Big\{1,\dfrac{1}{\sqrt{\pi(\nu-2)}}\Big\}\Big)^{1/2}, & \nu > 2,\\[2ex]
K_{\nu/2}\Big(\dfrac{2\nu\Delta^{(1)}}{\mu}\Big)^{\frac{\nu}{\nu+2}}, & 0<\nu\le2.
\end{cases}
\end{equation}}
{Indeed, in this case we have \(A_N=2\Delta^{(1)}\), \(B_N=4\Delta^{(2)}\),
\(\alpha\beta=\mu\)}.

Related Gamma approximation bounds for sums of independent nonnegative
random variables were obtained by \citet{LiuXia2021}, in particular through
their Corollary~3.4 for convolutions of independent Gamma variables. The
bounds below are different in spirit: they are tailored to the
Satterthwaite construction and express the error directly through the
discrepancy between the individual half-scales \(\tau_i\) and the matched
half-scale \(\tau=\mu/\nu\).

\subsection{Unequal-variance pooled variances}
\label{subsec:synthesis-variance}

We begin with the classical setting behind Welch--Satterthwaite type
approximations. Let \(K\) be fixed and suppose that, for \(k=1,\ldots,K\),
\(X_{k,1},\ldots,X_{k,n_k}\overset{\mathrm{iid}}{\sim}
N(0,\sigma_k^2)\), with the samples mutually independent across groups.
Consider
\begin{equation}
\label{eq:pooled-variance}
  W_n
  :=
  \sum_{k=1}^K \frac{S_k^2}{n_k},
  \qquad
  S_k^2
  :=
  \frac{1}{n_k-1}
  \sum_{i=1}^{n_k}
  \left(X_{k,i}-\overline X_k\right)^2 .
\end{equation}
Since
\[
  \frac{S_k^2}{n_k}
  \sim
  \frac{\sigma_k^2}{n_k(n_k-1)}\chi^2_{n_k-1}
  \sim
  \mathcal G_{(n_k-1)/2,\;2\sigma_k^2/(n_k(n_k-1))},
\]
the half-scales are \(\tau_{k,n}:=\sigma_k^2/(n_k(n_k-1))\). The
Satterthwaite approximation is
\(U_n\sim\mathcal G_{\nu_n/2,2\mu_n/\nu_n}\), where
\begin{equation}
\label{eq:satterthwaite-Un}
  \mu_n
  :=
  \sum_{k=1}^K\frac{\sigma_k^2}{n_k},
  \qquad
  \nu_n
  :=
  \frac{
    \left(\sum_{k=1}^K \sigma_k^2/n_k\right)^2
  }{
    \sum_{k=1}^K \sigma_k^4/(n_k^2(n_k-1))
  }.
\end{equation}
The matched half-scale is \(\tau_n=\mu_n/\nu_n\).

\begin{cor}
\label{cor:satterthwaite_steinweight}
Let \(W_n\) be defined by \eqref{eq:pooled-variance}, and let
\(U_n\sim\mathcal G_{\nu_n/2,2\mu_n/\nu_n}\), with \(\mu_n,\nu_n\) defined
by \eqref{eq:satterthwaite-Un}. For \(r=1,2\), set
\[
  \Delta_n^{(r)}
  :=
  \max_{1\le k\le K}
  \left\{
  \left|
  \frac{\tau_{k,n}}{\mu_n/\nu_n}-1
  \right|
  \tau_{k,n}^r
  \right\}.
\]
Then
\[
  d_1(W_n,U_n)\le 4\Delta_n^{(1)},\qquad
  d_2(W_n,U_n)
  \le
  8\Delta_n^{(2)}
  +
  2\mu_n e_{\nu_n/2}\Delta_n^{(1)}.
\]
{
Moreover,
\[
d_{\mathrm{Kol}}(W_n,U_n)\le
\begin{cases}
2\Big(\dfrac{\nu_n\Delta_n^{(1)}}{\mu_n}\min\Big\{1,\dfrac{1}{\sqrt{\pi(\nu_n-2)}}\Big\}\Big)^{1/2}, & \nu_n>2,\\[2ex]
K_{\nu_n/2}\Big(\dfrac{2\nu_n\Delta_n^{(1)}}{\mu_n}\Big)^{\frac{\nu_n}{\nu_n+2}}, & 0<\nu_n\le2,
\end{cases}
\]
with $K_a$ as in Proposition~\ref{prop:Kol-small-shape}.}
\end{cor}

\begin{rem}
\label{rem:satterthwaite-rate}
Assume that \(K\) is fixed and that \(n_k=c_{k,n}n\), where
\(n=\sum_{k=1}^K n_k\), \(\sum_{k=1}^K c_{k,n}=1\), and the proportions
\(c_{k,n}\) are bounded away from \(0\). This balanced regime ensures that
all group sizes grow at the same order; otherwise a small group may dominate
the largest scale discrepancy. If the variances \(\sigma_k^2\) are fixed,
then \(\mu_n\asymp n^{-1}\), \(\nu_n\asymp n\),
\(\tau_{k,n}\asymp n^{-2}\), and \(\tau_n\asymp n^{-2}\). Consequently
\(\Delta_n^{(1)}=O(n^{-2})\) and \(\Delta_n^{(2)}=O(n^{-4})\). Since
\(e_{\nu_n/2}=O(n^{-1/2})\), Corollary~\ref{cor:satterthwaite_steinweight}
gives
\[
  d_1(W_n,U_n)=O(n^{-2}),\qquad
  d_2(W_n,U_n)=O(n^{-7/2}),\qquad
  d_{\mathrm{Kol}}(W_n,U_n)=O(n^{-1/4}).
\]
For \(K=2\), this is the usual Welch--Satterthwaite approximation.
\end{rem}
 
\subsection{Sample variance of a stationary Gaussian time series}
\label{subsec:ts-variance}

The second example shows how serial dependence enters the approximation
through the spectrum of the covariance matrix. Let
\((X_t)_{t\in\mathbb Z}\) be a zero-mean stationary Gaussian process with
\(\Cov[X_t,X_{t+h}]=\gamma_h=\gamma_{-h}\), and assume
\(\sum_{h\in\mathbb Z}|\gamma_h|<\infty\). Its spectral density is
\[
  \F(\omega)
  :=
  \frac{1}{2\pi}
  \sum_{h\in\mathbb Z}\gamma_h e^{\i\omega h},
  \qquad
  \omega\in[-\pi,\pi].
\]
Write \(\sigma^2:=\gamma_0=\V[X_t]\). For
\(\bs X_n:=(X_1,\ldots,X_n)'\), we have
\(\bs X_n\sim N_n(0,\Gamma_n)\), where
\(\Gamma_n=(\gamma_{a-b})_{1\le a,b\le n}\). Write the normalized spectral
decomposition as
\[
  \Gamma_n=\sigma^2Q_n\Lambda_nQ_n',
  \qquad
  \Lambda_n=\mathrm{diag}(\lambda_{1,n},\ldots,\lambda_{n,n}),
  \qquad
  \sum_{j=1}^n\lambda_{j,n}=n.
\]
By Szeg\"o's theorem and its refinements, for every continuous \(g\),
\begin{equation}
\label{eq:szego}
  \frac{1}{n}\sum_{j=1}^n g(\lambda_{j,n})
  \to
  \frac{1}{2\pi}
  \int_{-\pi}^{\pi}
  g\left(\frac{2\pi\F(\omega)}{\sigma^2}\right)
  \,d\omega .
\end{equation}

Consider the normalized sample variance
\begin{equation}
\label{eq:normalized_variance}
  W_n
  :=
  \frac{1}{n\sigma^2}\sum_{t=1}^n X_t^2.
\end{equation}
If \(Z\sim N_n(0,I_n)\), then
\[
  W_n
  \cong
  \frac1n Z'Q_n\Lambda_nQ_n'Z
  \cong
  \sum_{j=1}^n
  \frac{\lambda_{j,n}}{n}\xi_j,
  \qquad
  \xi_1,\ldots,\xi_n\overset{\mathrm{iid}}{\sim}\chi_1^2.
\]
Thus \(\tau_{j,n}:=\lambda_{j,n}/n\). Since
\(\sum_j\lambda_{j,n}=n\), we have \(\mu_n=1\). Put
\begin{equation}
\label{eq:ts-Un}
  \omega_n:=\frac1n\sum_{j=1}^n\lambda_{j,n}^2,
  \qquad
  \nu_n:=\frac{n}{\omega_n},
  \qquad
  \tau_n:=\frac{\omega_n}{n}.
\end{equation}
The Satterthwaite approximation is
\(U_n\sim\mathcal G_{\nu_n/2,2/\nu_n}\).

\begin{cor}
\label{cor:variance_steinweight}
Let \(W_n\) be defined by \eqref{eq:normalized_variance}, and let
\(U_n\sim\mathcal G_{\nu_n/2,2/\nu_n}\), with \(\nu_n\) defined in
\eqref{eq:ts-Un}. For \(r=1,2\), set
\[
  \Delta_n^{(r)}
  :=
  \max_{1\le j\le n}
  \left\{
  \left|
    \frac{\lambda_{j,n}}{\omega_n}-1
  \right|
  \left(\frac{\lambda_{j,n}}{n}\right)^r
  \right\}.
\]
Then
\[
  d_1(W_n,U_n)\le 4\Delta_n^{(1)},\qquad
  d_2(W_n,U_n)
  \le
  8\Delta_n^{(2)}
  +
  2e_{\nu_n/2}\Delta_n^{(1)}.
\]
{
Moreover,
\[
d_{\mathrm{Kol}}(W_n,U_n)\le
\begin{cases}
2\Big(\dfrac{n\Delta_n^{(1)}}{\omega_n}\min\Big\{1,\dfrac{1}{\sqrt{\pi(\nu_n-2)}}\Big\}\Big)^{1/2}, & \nu_n> 2,\\[2ex]
K_{\nu_n/2}\Big(\dfrac{2n\Delta_n^{(1)}}{\omega_n}\Big)^{\frac{\nu_n}{\nu_n+2}}, & 0<\nu_n\le2,
\end{cases}
\]
with $K_a$ as in Proposition~\ref{prop:Kol-small-shape}.} 
\end{cor}

\begin{rem}
\label{rem:ts-rate}
The quantity \(\omega_n\) measures the spectral dispersion of the normalized
covariance matrix. By \eqref{eq:szego} with \(g(x)=x^2\), one has
\(\omega_n=O(1)\), while Cauchy--Schwarz gives \(\omega_n\ge1\). Thus
\(\nu_n=n/\omega_n\asymp n\). Moreover,
\[
  \max_{1\le j\le n}\lambda_{j,n}
  \le
  \frac{1}{\sigma^2}
  \sum_{h\in\mathbb Z}|\gamma_h|.
\]
It follows that \(\Delta_n^{(1)}=O(n^{-1})\) and
\(\Delta_n^{(2)}=O(n^{-2})\). Corollary~\ref{cor:variance_steinweight}
therefore gives
\[
  d_1(W_n,U_n)=O(n^{-1}),\qquad
  d_2(W_n,U_n)=O(n^{-3/2}),\qquad
  d_{\mathrm{Kol}}(W_n,U_n)=O(n^{-1/4}).
\]
\end{rem}

An illustration is given in Figure~\ref{fig:ts-wass-empirical}. We consider
stationary Gaussian AR(1) processes
\[
  X_t=\rho X_{t-1}+\varepsilon_t,
  \qquad
  \varepsilon_t\sim N(0,1-\rho^2),
  \qquad |\rho|<1,
\]
so that \(\V[X_t]=1\). Hence \(W_n\) is the normalized sample variance in
Corollary~\ref{cor:variance_steinweight}. For \(n=10,20,\ldots,200\), we
approximate \(d_1(W_n,U_n)\) by Monte Carlo simulation with
\(M=1,\!000,\!000\) replications. We also compare the Satterthwaite
approximation with the normal approximation \(Z_n\sim N(1,2/\nu_n)\).

\begin{figure}[!ht]
    \centering
    \includegraphics[width=0.49\linewidth]{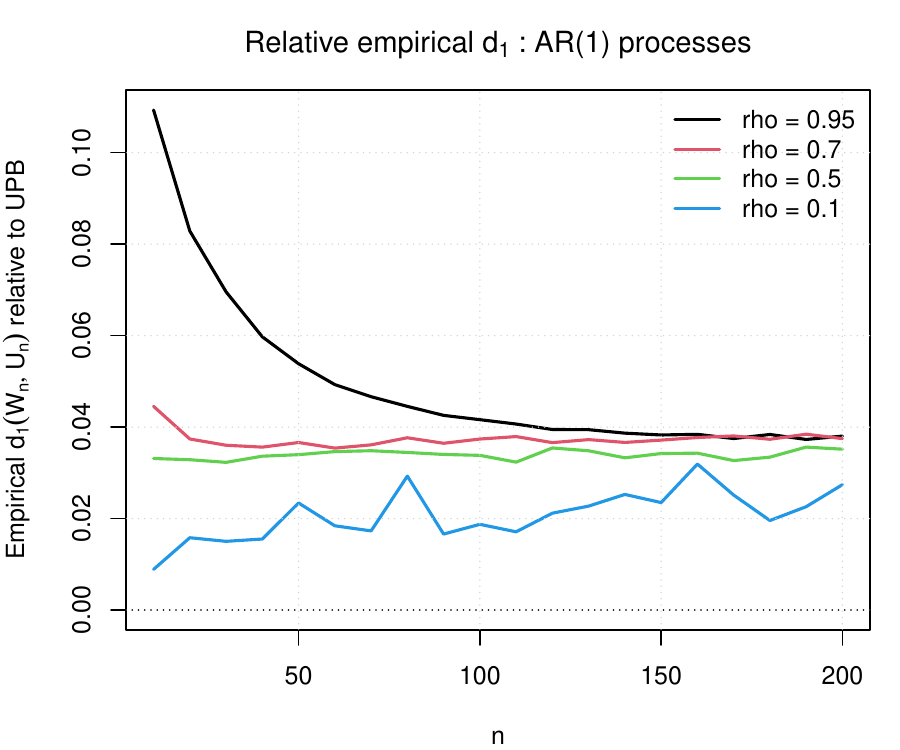}
    \includegraphics[width=0.49\linewidth]{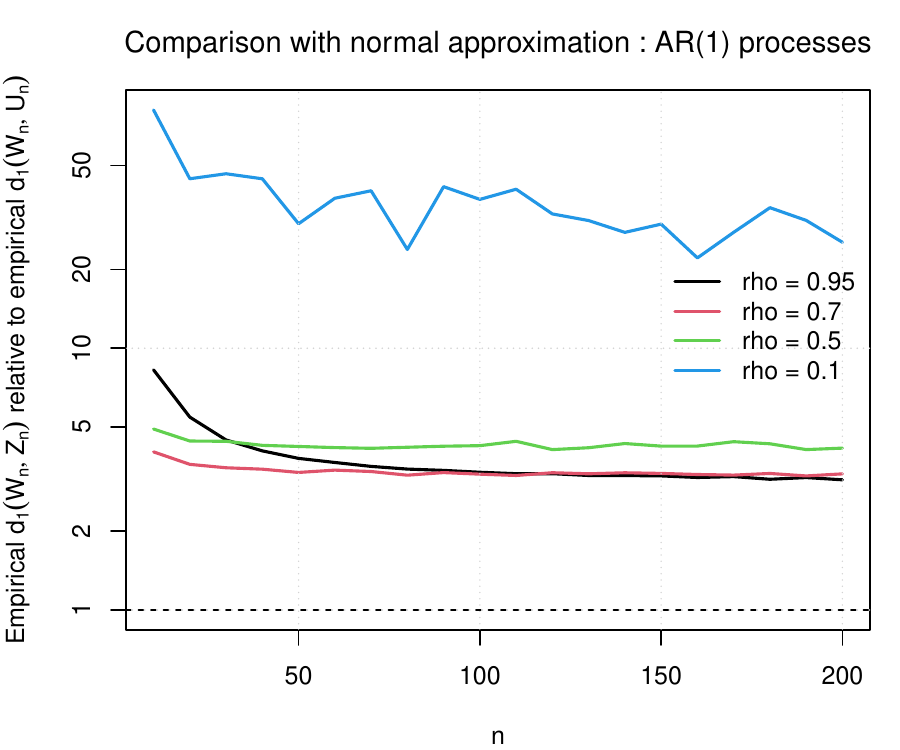}
    \caption{{\it Left: empirical \(d_1(W_n,U_n)\) relative to the upper bound from
    Corollary~\ref{cor:variance_steinweight}. Right: empirical
    \(d_1(W_n,Z_n)\) relative to empirical \(d_1(W_n,U_n)\).}}
    \label{fig:ts-wass-empirical}
\end{figure}

The left panel shows that the theoretical upper bound has the correct scale
in this example: the empirical Wasserstein distance remains below the bound
and the ratio is fairly stable over the displayed range of sample sizes. The
effect of dependence is visible. Larger values of \(|\rho|\) produce a more
heterogeneous covariance spectrum and therefore a larger scale-discrepancy
term \(\Delta_n^{(1)}\).

The right panel compares the Satterthwaite approximation with the normal
approximation having the same first two moments. The ratios are systematically
larger than one, indicating that the Gamma approximation is closer in
Wasserstein distance throughout the simulation. This is natural here:
\(W_n\) is a positive quadratic form, and the Gamma approximation preserves
both positivity and the skewed finite-sample shape, whereas the normal
approximation is symmetric. When \(\rho=0\), the covariance spectrum is flat
and the Satterthwaite approximation is exact.

\subsection{Traces of Wishart random matrices}
\label{subsec:trace-wishart}

We next consider trace functionals of Wishart matrices. This example is
useful because the same Satterthwaite mechanism covers both the standard
one-sample Wishart model and sums of independent Wishart matrices with
different covariance structures.

Let \(\mathfrak W_n\sim\W_d(n,n^{-1}\Sigma)\), where \(\Sigma\) is symmetric
positive definite, and define \(W_n:=\tr(\mathfrak W_n)\). If
\(\Sigma=\sigma^2I_d\), then \(W_n\sim(\sigma^2/n)\chi^2_{nd}\). For general
\(\Sigma\), writing \(\lambda_1,\ldots,\lambda_d\) for its eigenvalues, the
Laplace transform gives \citep{Khatri1971}
\[
  W_n
  \cong
  \sum_{j=1}^d
  \frac{\lambda_j}{n}\xi_j,
  \qquad
  \xi_1,\ldots,\xi_d\overset{\mathrm{iid}}{\sim}\chi_n^2.
\]
Thus \(\tau_{j,n}:=\lambda_j/n\). Here
\begin{equation}
\label{eq:trace-Un}
  \mu=\tr(\Sigma),
  \qquad
  \nu_n
  =
  n\frac{\tr(\Sigma)^2}{\tr(\Sigma^2)},
  \qquad
  \tau_n
  =
  \frac{\tr(\Sigma^2)}{n\tr(\Sigma)}.
\end{equation}
The Satterthwaite approximation is
\(U_n\sim\mathcal G_{\nu_n/2,2\tr(\Sigma)/\nu_n}\).

\begin{cor}
\label{cor:trace_steinweight}
Let \(W_n=\tr(\mathfrak W_n)\), with
\(\mathfrak W_n\sim\W_d(n,n^{-1}\Sigma)\), and let
\(U_n\sim\mathcal G_{\nu_n/2,2\tr(\Sigma)/\nu_n}\), with \(\nu_n\) defined in
\eqref{eq:trace-Un}. For \(r=1,2\), set
\[
  \Delta_n^{(r)}
  :=
  \max_{1\le j\le d}
  \left\{
  \left|
    \frac{\lambda_j\tr(\Sigma)}
         {\tr(\Sigma^2)}
    -1
  \right|
  \left(\frac{\lambda_j}{n}\right)^r
  \right\}.
\]
Then
\[
  d_1(W_n,U_n)\le 4\Delta_n^{(1)},\qquad
  d_2(W_n,U_n)
  \le
  8\Delta_n^{(2)}
  +
  2\tr(\Sigma)e_{\nu_n/2}\Delta_n^{(1)}.
\]
{
Moreover,
\[
d_{\mathrm{Kol}}(W_n,U_n)\le
\begin{cases}
2\Big(n\Delta_n^{(1)}\dfrac{\tr(\Sigma)}{\tr(\Sigma^2)}\min\Big\{1,\dfrac{1}{\sqrt{\pi(\nu_n-2)}}\Big\}\Big)^{1/2}, & \nu_n>2,\\[2ex]
K_{\nu_n/2}\Big(2n\Delta_n^{(1)}\dfrac{\tr(\Sigma)}{\tr(\Sigma^2)}\Big)^{\frac{\nu_n}{\nu_n+2}}, & 0<\nu_n\le2,
\end{cases}
\]
with $K_a$ as in Proposition~\ref{prop:Kol-small-shape}.}
 
\end{cor}

\begin{rem}
\label{rem:trace_rate}
Assume that \(0<c\le\lambda_j\le C<\infty\), \(j=1,\ldots,d\), with constants
independent of \(d\). Then \(\tr(\Sigma)\asymp d\),
\(\tr(\Sigma^2)\asymp d\), \(\nu_n\asymp nd\),
\(\Delta_n^{(1)}=O(n^{-1})\), and \(\Delta_n^{(2)}=O(n^{-2})\). Since
\(e_{\nu_n/2}=O((nd)^{-1/2})\), Corollary~\ref{cor:trace_steinweight} gives
\[
  d_1(W_n,U_n)=O(n^{-1}),\qquad
  d_2(W_n,U_n)=O(n^{-3/2}d^{1/2})+O(n^{-2}),\qquad
  d_{\mathrm{Kol}}(W_n,U_n)=O((nd)^{-1/4}).
\]
Thus increasing the dimension improves the Kolmogorov bound through the
effective degrees of freedom \(\nu_n\), while the \(d_2\)-bound also carries
the natural size factor \(\tr(\Sigma)\asymp d\). If \(d\asymp n\), the
Kolmogorov rate is of Berry--Esseen order in \(n\).
\end{rem}

The same calculation extends directly to sums of independent Wishart
matrices. Let \(K\in\mathbb N\) be fixed, let
\(n_1,\ldots,n_K\in\mathbb N\), set \(n:=\sum_{k=1}^K n_k\), and consider
\begin{equation}
\label{eq:linearcomb-trace}
  \mathsf W_n
  =
  \sum_{k=1}^K \mathfrak W_{k,n},
  \qquad
  \mathfrak W_{k,n}
  \sim
  \W_d(n_k,n_k^{-1}\Sigma_k),
\end{equation}
with the summands independent. A Satterthwaite approximation for such sums
was proposed by \citet{Tan1983}, based on matching mean and variance through
the generalized variance, that is, the determinant of the
\(d(d+1)/2\times d(d+1)/2\) covariance matrix of \(\mathsf W_n\). Here we
consider only the trace. If \(\lambda_{k1},\ldots,\lambda_{kd}\) are the
eigenvalues of \(\Sigma_k\), then
\[
  W_n:=\tr(\mathsf W_n)
  \cong
  \sum_{k=1}^K\sum_{l=1}^d
  \frac{\lambda_{kl}}{n_k}\xi_{kl},
  \qquad
  \xi_{kl}\sim\chi_{n_k}^2
\]
independently. Hence \(\tau_{kl}:=\lambda_{kl}/n_k\). Write
\begin{equation}
\label{eq:linearcomb-trace-Un}
  \mu
  =
  \sum_{k=1}^K\tr(\Sigma_k),
  \qquad
  \nu
  =
  \frac{
    \left(\sum_{k=1}^K\tr(\Sigma_k)\right)^2
  }{
    \sum_{k=1}^K n_k^{-1}\tr(\Sigma_k^2)
  },
  \qquad
  \tau=\frac{\mu}{\nu}.
\end{equation}
The Satterthwaite approximation is
\(U_n\sim\mathcal G_{\nu/2,2\mu/\nu}\).

\begin{cor}
\label{cor:linearcomb-trace_steinweight}
Let \(W_n=\tr(\mathsf W_n)\), with \(\mathsf W_n\) defined in
\eqref{eq:linearcomb-trace}, and let
\(U_n\sim\mathcal G_{\nu/2,2\mu/\nu}\), with \(\mu,\nu\) defined in
\eqref{eq:linearcomb-trace-Un}. For \(r=1,2\), set
\[
  \Delta_n^{(r)}
  :=
  \max_{\substack{1\le k\le K\\1\le l\le d}}
  \left\{
  \left|
  \frac{\lambda_{kl}/n_k}{\mu/\nu}-1
  \right|
  \left(\frac{\lambda_{kl}}{n_k}\right)^r
  \right\}.
\]
Then
\[
  d_1(W_n,U_n)\le 4\Delta_n^{(1)},\qquad
  d_2(W_n,U_n)
  \le
  8\Delta_n^{(2)}
  +
  2\mu e_{\nu/2}\Delta_n^{(1)}.
\]
{
Moreover,
\[
d_{\mathrm{Kol}}(W_n,U_n)\le
\begin{cases}
2\Big(\dfrac{\nu\Delta_n^{(1)}}{\mu}\min\Big\{1,\dfrac{1}{\sqrt{\pi(\nu-2)}}\Big\}\Big)^{1/2}, & \nu>2,\\[2ex]
K_{\nu/2}\Big(\dfrac{2\nu\Delta_n^{(1)}}{\mu}\Big)^{\frac{\nu}{\nu+2}}, & 0<\nu\le2,
\end{cases}
\]
with $K_a$ as in Proposition~\ref{prop:Kol-small-shape}}

\end{cor}

\begin{rem}
\label{rem:linearcomb-trace_rate}
Assume that \(K\) is fixed, \(n_k=c_{k,n}n\) with
\(\sum_{k=1}^K c_{k,n}=1\), and the \(c_{k,n}\) are bounded away from \(0\).
Assume also that \(0<c\le\lambda_{kl}\le C<\infty\) uniformly in
\(k,l,n,d\). Then \(\mu\asymp d\), \(\nu\asymp nd\),
\(\tau_{kl}\asymp n^{-1}\), \(\tau\asymp n^{-1}\),
\(\Delta_n^{(1)}=O(n^{-1})\), and \(\Delta_n^{(2)}=O(n^{-2})\). Since
\(e_{\nu/2}=O((nd)^{-1/2})\),
Corollary~\ref{cor:linearcomb-trace_steinweight} gives
\[
  d_1(W_n,U_n)=O(n^{-1}),\qquad
  d_2(W_n,U_n)=O(n^{-3/2}d^{1/2})+O(n^{-2}),\qquad
  d_{\mathrm{Kol}}(W_n,U_n)=O((nd)^{-1/4}).
\]
If \(d\) is fixed, the Kolmogorov rate reduces to \(O(n^{-1/4})\); if \(d\)
grows with \(n\), the effective degrees of freedom increase accordingly.
\end{rem}

\subsection{High-dimensional repeated-measures designs}
\label{subsec:repeated-anova}

We finally consider an ANOVA-type statistic for high-dimensional repeated
measurements. The distinctive feature of this example is that, in the oracle
case where \(\Sigma\) is known, the quality of the Gamma approximation is
governed by the rank \(q\) of the hypothesis projection rather than directly
by the number of subjects.

Let \(\mathbf H\in\mathbb R^{q\times d}\) be a full-rank hypothesis matrix,
and let \(\bs X_i\sim N_d(\bs\mu,\Sigma)\), \(i=1,\ldots,n\), be independent
repeated-measurement vectors. The null hypothesis is \(H_0:\mathbf H\bs\mu=0_q\),
or equivalently \(H_0:\mathbf T\bs\mu=0_d\), where
\[
  \mathbf T
  =
  \mathbf H'(\mathbf H\mathbf H')^+\mathbf H.
\]
Here \((\cdot)^+\) denotes the Moore--Penrose inverse. The matrix
\(\mathbf T\) is the orthogonal projection onto the hypothesis space
generated by \(\mathbf H\), and \(\mathrm{rk}(\mathbf T)=q\).

When the number of repeated measurements exceeds the number of subjects,
classical Wald-type tests cannot be used because they require inversion of
the covariance matrix. In this setting, \citet{RaufAhmad2008} proposed an
ANOVA-type statistic whose null distribution is approximated by a
Satterthwaite approximation. In the oracle setting where \(\Sigma\) is known,
the statistic is
\begin{equation}
\label{eq:anova-statistic}
  W_n
  =
  \frac{n}{\tr(\mathbf T\Sigma)}
  \overline{\bs X}'\mathbf T\overline{\bs X}.
\end{equation}
{Under \(H_0\), \(\mathbf T\bs\mu=0\), so
\[
  \overline{\bs X}'\mathbf T\overline{\bs X}
  =(\overline{\bs X}-\bs\mu)'\mathbf T
    (\overline{\bs X}-\bs\mu).
\]
Since \(\sqrt n(\overline{\bs X}-\bs\mu)\sim N_d(0,\Sigma)\),}
a spectral decomposition gives
\[
  W_n
  \cong
  \sum_{j=1}^q
  \frac{\lambda_j}{\tr(\mathbf T\Sigma)}\xi_j,
  \qquad
  \xi_1,\ldots,\xi_q\overset{\mathrm{iid}}{\sim}\chi_1^2,
\]
where \(\lambda_1,\ldots,\lambda_q\) are the non-zero eigenvalues of
\(\Sigma^{1/2}\mathbf T\Sigma^{1/2}\). Since
\(\sum_{j=1}^q\lambda_j=\tr(\mathbf T\Sigma)\) and
\(\sum_{j=1}^q\lambda_j^2=\tr((\mathbf T\Sigma)^2)\), the half-scales are
\(\tau_j:=\lambda_j/\tr(\mathbf T\Sigma)\). Thus \(\mu=1\) and
\begin{equation}
\label{eq:anova-Un}
  \nu
  =
  \frac{
    \tr(\mathbf T\Sigma)^2
  }{
    \tr((\mathbf T\Sigma)^2)
  },
  \qquad
  \tau=\frac1{\nu}.
\end{equation}
The Satterthwaite approximation is \(U_q\sim\mathcal G_{\nu/2,2/\nu}\).

\begin{cor}
\label{cor:anova_steinweight}
Let \(W_n\) be defined in \eqref{eq:anova-statistic}, and let
\(U_q\sim\mathcal G_{\nu/2,2/\nu}\), with \(\nu\) defined in
\eqref{eq:anova-Un}. For \(r=1,2\), set
\[
  \Delta_q^{(r)}
  :=
  \max_{1\le j\le q}
  \left\{
  \left|
  \frac{\lambda_j\tr(\mathbf T\Sigma)}
       {\tr((\mathbf T\Sigma)^2)}
  -1
  \right|
  \left(
  \frac{\lambda_j}{\tr(\mathbf T\Sigma)}
  \right)^r
  \right\}.
\]
Then
\[
  d_1(W_n,U_q)\le 4\Delta_q^{(1)},\qquad
  d_2(W_n,U_q)
  \le
  8\Delta_q^{(2)}
  +
  2e_{\nu/2}\Delta_q^{(1)}.
\]
{
Moreover,
\[
d_{\mathrm{Kol}}(W_n,U_q)\le
\begin{cases}
2\Big(\nu\Delta_q^{(1)}\min\Big\{1,\dfrac{1}{\sqrt{\pi(\nu-2)}}\Big\}\Big)^{1/2}, & \nu>2,\\[2ex]
K_{\nu/2}\big(2\nu\Delta_q^{(1)}\big)^{\frac{\nu}{\nu+2}}, & 0<\nu\le2,
\end{cases}
\]
with $K_a$ as in Proposition~\ref{prop:Kol-small-shape}.}
 
\end{cor}

\begin{rem}
\label{rem:anova_rate}
Assume that the non-zero eigenvalues of
\(\Sigma^{1/2}\mathbf T\Sigma^{1/2}\) satisfy
\(0<c\le\lambda_j\le C<\infty\), \(j=1,\ldots,q\). This condition is implied,
for instance, by corresponding uniform eigenvalue bounds on \(\Sigma\), since
\(\mathbf T\) is an orthogonal projection of rank \(q\). Then
\(\tr(\mathbf T\Sigma)\asymp q\), \(\tr((\mathbf T\Sigma)^2)\asymp q\),
\(\nu\asymp q\), \(\Delta_q^{(1)}=O(q^{-1})\), and
\(\Delta_q^{(2)}=O(q^{-2})\). Since \(e_{\nu/2}=O(q^{-1/2})\),
Corollary~\ref{cor:anova_steinweight} gives
\[
  d_1(W_n,U_q)=O(q^{-1}),\qquad
  d_2(W_n,U_q)=O(q^{-3/2}),\qquad
  d_{\mathrm{Kol}}(W_n,U_q)=O(q^{-1/4}).
\]
If \(q=q_n\) grows with \(n\), these rates are read along the chosen
high-dimensional regime. For instance, \(q_n\asymp n\) yields
\[
  d_1(W_n,U_q)=O(n^{-1}),\qquad
  d_2(W_n,U_q)=O(n^{-3/2}),\qquad
  d_{\mathrm{Kol}}(W_n,U_q)=O(n^{-1/4}).
\]
In many high-dimensional repeated-measures regimes, the sample size, the
number of repeated measurements, and the rank of the hypothesis projection
are assumed to grow at comparable rates.
\end{rem}

\section{Outlook}
\label{sec:outlook}

The results of this paper show that the Satterthwaite approximation
can be studied directly on its natural Gamma scale, without passing
through distribution-function expansions or asymptotic
normalizations. The kernel representations developed above make the
relevant Stein factors explicit and lead to bounds whose dependence on
the individual scale parameters is transparent. Several extensions
seem worth pursuing.

First, the non-uniform Stein factors obtained here could be used to
sharpen other Gamma approximation results. In normal approximation,
non-uniform bounds often capture local features that are invisible in
global estimates; see \cite[Chapter~8]{Chen2011}. The Gamma kernels
derived in this paper suggest that an analogous refinement is possible
for positive approximations, especially in regimes where the target
shape parameter is large.

Second, it would be natural to develop higher-order versions of the
present kernel calculus. The bounds of \cite{GauntThesis,Gaunt2017}
indicate that Gamma--Stein estimates should persist at arbitrary
derivative order, but the representations used here are currently
tailored to the first two derivatives needed for moment
matching. Extending the construction to higher orders could provide a
Stein-based treatment of higher-order moment-matching approximations,
including Wood's \(F\)-approximation \cite{Wood1989}.

Third, the present kernels are intrinsically adapted to the positive
half-line.  As discussed in Remark~\ref{rem:comparison-gpr}, this
makes the approach different from whole-line Gamma Stein equations
such as those studied in \cite{Dobler2018}. A natural next step is to
understand whether analogous one-sided kernel representations can be
built for signed targets or for approximations involving both positive
and negative components.

Finally, the method should apply beyond Gamma summands. The proof of
Theorem~\ref{thm:wass_gamma_stein-gamma} separates two ingredients: a
Gamma--Stein transfer identity for the target law, and
integration-by-parts identities for the summands. Replacing the latter
by suitable Stein identities for other positive variables would lead
to Satterthwaite-type bounds for linear combinations of independent
squared variables that are not necessarily chi-square distributed.

\section*{Acknowledgements}

{The authors are grateful to the referee for their insightful and precise
comments, which helped improve several parts of the paper.
We are particularly grateful to the referee for suggesting the
smoothing argument used in the proof of
Proposition~\ref{prop:Kol-small-shape}, which allowed us to extend
the Kolmogorov bounds to Gamma targets with shape parameter below one.}

\section*{Declaration on the use of AI-assisted tools}

{The authors used ChatGPT (OpenAI, GPT-5.5 Thinking,
  accessed in May 2026) in preparing the original submission, and
  ChatGPT (OpenAI, GPT-6.2 Astra) and Claude (Anthropic, Opus), both
  accessed in November 2026, in preparing the revision. These tools
  were used for language editing, to clarify the exposition, and to
  assist with the numerical computations for
  Propositions~\ref{prop:Ca-bounds} and~\ref{prop:ea-explicit}. During
  the revision, the authors also used these tools to check the
  smoothing argument in Proposition~\ref{prop:Kol-small-shape} for
  Gamma targets with shape parameter \(0<\alpha\leq1\) and to clarify
  the dependence of the bounds in
  Lemma~\ref{lem:condition_rate_discrepancy} on the model
  parameters. The authors used ChatGPT to derive the integral identity
  showing that \(\widetilde R_n(0)\) decreases with \(n\) in the proof
  of Proposition~\ref{prop:ea-explicit}.

The authors reviewed and edited all AI-assisted material and take full
responsibility for the final manuscript.}

\appendix

\section{Proofs from Section \ref{sec:steingamma}}
\label{appendix:steinstuff}

\begin{proof}[Proof of Lemma \ref{lem:gamma-mills-ratios}]
By scaling, it is enough to prove the result for \(b=1\). 
For \(y>0\),
we get 
\[
  \underline M_{a,1}(y)
  =
  \int_0^1 t^{a-1}e^{y(1-t)}\,dt.
\]
Since the integration range is bounded, differentiation under the integral
sign is immediate. Hence, for every \(n\in\mathbb N\) and every \(y>0\),
\[
  \underline M_{a,1}^{(n)}(y)
  =
  \int_0^1 (1-t)^n t^{a-1}e^{y(1-t)}\,dt
  \ge 0.
\]
Thus \(y\mapsto \underline M_{a,1}(y)\) is absolutely monotone on
\((0,\infty)\).

Similarly, with the change of variables \(u=yt\), we obtain
\[
  \overline M_{a,1}(y)
  =
  \int_1^\infty t^{a-1}e^{-y(t-1)}\,dt.
\]
For every compact interval \(I\subset(0,\infty)\) and every \(n\in\mathbb N\),
the function
\[
  t\mapsto (t-1)^n t^{a-1} e^{-\eta(t-1)},
  \qquad \eta:=\inf I>0,
\]
is integrable on \((1,\infty)\). Hence differentiation under the integral
sign is justified locally uniformly in \(y\in I\). It follows that, for every
\(n\in\mathbb N\) and every \(y>0\),
\[
  \overline M_{a,1}^{(n)}(y)
  =
  (-1)^n
  \int_1^\infty (t-1)^n t^{a-1}e^{-y(t-1)}\,dt.
\]
Consequently, \(y\mapsto \overline M_{a,1}(y)\) is completely monotone on
\((0,\infty)\).  
\end{proof}

\begin{proof}[Proof of Lemma \ref{lma:boundskern}]
Recall that
\[
  \underline{\mathrm P}^{1}_{a,b}(y)
  =
  \int_0^y P_{a,b}(u)\,du,
  \qquad
  \overline{\mathrm P}^{1}_{a,b}(y)
  =
  \int_y^\infty \overline P_{a,b}(u)\,du .
\]
We shall use
\begin{equation}
\label{eq:P1-identities}
  \underline{\mathrm P}^{1}_{a,b}(y)
  =
  (y-ab)P_{a,b}(y)+by\,p_{a,b}(y),
  \qquad
  \overline{\mathrm P}^{1}_{a,b}(y)
  =
  (ab-y)\overline P_{a,b}(y)+by\,p_{a,b}(y),
\end{equation}
which follows by integration by parts from
\[
  (bxp_{a,b}(x))'=(ab-x)p_{a,b}(x).
\]

Since the two pieces in \(K^{0,1}_{a,b}\) have disjoint supports and
\(\underline M_{a,b},\overline M_{a,b}\ge0\),
\[
  \mathbb E[|K^{0,1}_{a,b}(X,y)|]
  =
  \overline M_{a,b}(y)\underline{\mathrm P}^{1}_{a,b}(y)
  +
  \underline M_{a,b}(y)\overline{\mathrm P}^{1}_{a,b}(y).
\]
Using \eqref{eq:P1-identities}, we obtain
\[
  \mathbb E[|K^{0,1}_{a,b}(X,y)|]
  =
  \overline P_{a,b}(y)+P_{a,b}(y)=1.
\]

Then, using the signs from Lemma~\ref{lem:gamma-mills-ratios},
\[
  \mathbb E[|K^{1,1}_{a,b}(X,y)|]
  =
  -\overline M_{a,b}'(y)\underline{\mathrm P}^{1}_{a,b}(y)
  +
  \underline M_{a,b}'(y)\overline{\mathrm P}^{1}_{a,b}(y) = C^1_{a,b}(y).
\]

Introduce the second iterated tails
\begin{equation}
    I_{a,b}(y)
  :=
  \int_0^y \underline{\mathrm P}^{1}_{a,b}(x)\,dx,
  \qquad
  J_{a,b}(y)
  :=
  \int_y^\infty \overline{\mathrm P}^{1}_{a,b}(x)\,dx. \label{eq:ij}
\end{equation}
A second integration by parts yields
\begin{equation}
\label{eq:IJ-identities}
  I_{a,b}(y)
  =
  \frac12
  \left\{
    (y-ab-b)\underline{\mathrm P}^{1}_{a,b}(y)
    +
    byP_{a,b}(y)
  \right\},
\end{equation}
and
\begin{equation}
\label{eq:IJ-identities-survival}
  J_{a,b}(y)
  =
  \frac12
  \left\{
    (ab+b-y)\overline{\mathrm P}^{1}_{a,b}(y)
    +
    by\overline P_{a,b}(y)
  \right\}.
\end{equation}
Now, direct differentiation gives
\begin{equation}
\label{eq:Mprime-identities}
  \underline M_{a,b}'(y)
  =
  \frac{\underline{\mathrm P}^{1}_{a,b}(y)}
       {b^2y^2p_{a,b}(y)},
  \qquad
  \overline M_{a,b}'(y)
  =
  -\frac{\overline{\mathrm P}^{1}_{a,b}(y)}
        {b^2y^2p_{a,b}(y)}.
\end{equation}
Differentiating one more time, and using
\eqref{eq:IJ-identities}--\eqref{eq:IJ-identities-survival}, gives
\begin{equation}
\label{eq:Msecond-identities}
  \underline M_{a,b}''(y)
  =
  \frac{2I_{a,b}(y)}{b^3y^3p_{a,b}(y)},
  \qquad
  \overline M_{a,b}''(y)
  =
  \frac{2J_{a,b}(y)}{b^3y^3p_{a,b}(y)}.
\end{equation}
Hence by \eqref{eq:IJ-identities},  \eqref{eq:IJ-identities-survival} and \eqref{eq:Msecond-identities}, 
\begin{align*}
     \mathbb E[|K^{2,1}_{a,b}(X,y)|]
  & = 
  \overline M_{a,b}''(y)\underline{\mathrm P}^{1}_{a,b}(y)
  +
  \underline M_{a,b}''(y)\overline{\mathrm P}^{1}_{a,b}(y) 
  =
  \frac1{by}
\end{align*}

For \(K^{2,2}_{a,b}\), since
\(\underline M_{a,b}''(y),\overline M_{a,b}''(y)\ge0\),
\[
  \mathbb E[|K^{2,2}_{a,b}(X,y)|]
  =
  \overline M_{a,b}''(y)
  \int_0^y \underline{\mathrm P}^{1}_{a,b}(x)\,dx
  +
  \underline M_{a,b}''(y)
  \int_y^\infty \overline{\mathrm P}^{1}_{a,b}(x)\,dx
  =
  C^2_{a,b}(y).
\]

Finally, differentiating \eqref{eq:Msecond-identities} gives
\[
  \underline M_{a,b}'''(y)
  =
  \frac{2}{b^3y^3p_{a,b}(y)}
  \left\{
    \underline{\mathrm P}^{1}_{a,b}(y)
    -
    \left(\frac{a+2}{y}-\frac1b\right)I_{a,b}(y)
  \right\},
\]
and
\[
  \overline M_{a,b}'''(y)
  =
  \frac{2}{b^3y^3p_{a,b}(y)}
  \left\{
    -\overline{\mathrm P}^{1}_{a,b}(y)
    -
    \left(\frac{a+2}{y}-\frac1b\right)J_{a,b}(y)
  \right\}.
\]
Using again the signs from Lemma~\ref{lem:gamma-mills-ratios},
\[
  \mathbb E[|K^{3,2}_{a,b}(X,y)|]
  =
  -\overline M_{a,b}'''(y)I_{a,b}(y)
  +
  \underline M_{a,b}'''(y)J_{a,b}(y) =
  \frac1{by},
\]
by the identity already proved above. This completes the proof.
\end{proof}

\begin{proof}[Proof of Lemma \ref{lem:solution}]
We first prove the assertions for smooth compactly supported \(h\); the
general case follows by approximation, using the bounds obtained below.

Expanding the kernel in \eqref{eq:fh-explicit}, we get, for every \(x>0\),
\begin{equation}
\label{eq:fh-proof-start}
  f_{h;a,b}(x)
  =
  -\frac{\overline P_{a,b}(x)}{bxp_{a,b}(x)}
  \int_0^x P_{a,b}(t)h'(t)\,dt
  -
  \frac{P_{a,b}(x)}{bxp_{a,b}(x)}
  \int_x^\infty \overline P_{a,b}(t)h'(t)\,dt .
\end{equation}
For brevity, set
\[
  \underline{\Phi}_h(x):=\int_0^x P_{a,b}(t)h'(t)\,dt,
  \qquad
  \overline{\Phi}_h(x):=\int_x^\infty \overline P_{a,b}(t)h'(t)\,dt .
\]
Then
\[
  f_{h;a,b}(x)
  =
  -\overline M_{a,b}(x)\underline\Phi_h(x)
  -
  \underline M_{a,b}(x)\overline{\Phi}_h(x).
\]
Since
\[
  \underline\Phi_h'(x)=P_{a,b}(x)h'(x),
  \qquad
  \overline{\Phi}_h'(x)=-\overline P_{a,b}(x)h'(x),
\]
differentiating gives
\[
\begin{aligned}
  f_{h;a,b}'(x)
  &=
  -\overline M_{a,b}'(x)\underline\Phi_h(x)
  -
  \underline M_{a,b}'(x)\overline{\Phi}_h(x)  \\
  &\quad
  -
  \overline M_{a,b}(x)P_{a,b}(x)h'(x)
  +
  \underline M_{a,b}(x)\overline P_{a,b}(x)h'(x).
\end{aligned}
\]
The last two terms cancel, because
\[
  \overline M_{a,b}(x)P_{a,b}(x)
  =
  \underline M_{a,b}(x)\overline P_{a,b}(x)
  =
  \frac{P_{a,b}(x)\overline P_{a,b}(x)}
       {bxp_{a,b}(x)} .
\]
Therefore
\[
  f_{h;a,b}'(x)
  =
  -\overline M_{a,b}'(x)\underline\Phi_h(x)
  -
  \underline M_{a,b}'(x)\overline{\Phi}_h(x)
  =
  \mathbb E\left[
    h'(X)K_{a,b}^{1,1}(X,x)
  \right]
  =
  g_{h;a,b}(x).
\]

We now check the Stein equation. By integration by parts,
\[
  \underline\Phi_h(x)
  =
  P_{a,b}(x)h(x)-\int_0^x h(t)p_{a,b}(t)\,dt,
\]
and, since \(h\) is compactly supported,
\[
  \overline{\Phi}_h(x)
  =
  -\overline P_{a,b}(x)h(x)
  +
  \int_x^\infty h(t)p_{a,b}(t)\,dt .
\]
It follows that
\[
  \overline P_{a,b}(x)\underline\Phi_h(x)
  +
  P_{a,b}(x)\overline{\Phi}_h(x)
  =
  P_{a,b}(x)\mathcal G_{a,b}h
  -
  \int_0^x h(t)p_{a,b}(t)\,dt .
\]
Hence \eqref{eq:fh-proof-start} is equivalently
\[
  f_{h;a,b}(x)
  =
  \frac{1}{bxp_{a,b}(x)}
  \int_0^x
  \left\{h(t)-\mathcal G_{a,b}h\right\}
  p_{a,b}(t)\,dt .
\]
Using
\[
  (bxp_{a,b}(x))'=(ab-x)p_{a,b}(x),
\]
we obtain
\[
  bxf_{h;a,b}'(x)-(x-ab)f_{h;a,b}(x)
  =
  h(x)-\mathcal G_{a,b}h.
\]
Thus \(f_{h;a,b}\) solves the Gamma--Stein equation.

Next, differentiating the expression
\[
  g_{h;a,b}(x)
  =
  -\overline M_{a,b}'(x)\underline\Phi_h(x)
  -
  \underline M_{a,b}'(x)\overline{\Phi}_h(x)
\]
gives
\[
\begin{aligned}
  g_{h;a,b}'(x)
  &=
  -\overline M_{a,b}''(x)\underline\Phi_h(x)
  -
  \underline M_{a,b}''(x)\overline{\Phi}_h(x) \\
  &\quad
  -
  \overline M_{a,b}'(x)P_{a,b}(x)h'(x)
  +
  \underline M_{a,b}'(x)\overline P_{a,b}(x)h'(x).
\end{aligned}
\]
The first line is
\[
  \mathbb E\left[
    h'(X)K_{a,b}^{2,1}(X,x)
  \right].
\]
For the diagonal contribution, use
\[
  \underline M_{a,b}'(x)
  =
  \frac{\underline{\mathrm P}_{a,b}^{1}(x)}
       {b^2x^2p_{a,b}(x)},
  \qquad
  \overline M_{a,b}'(x)
  =
  -\frac{\overline{\mathrm P}_{a,b}^{1}(x)}
        {b^2x^2p_{a,b}(x)}.
\]
Then
\[
\begin{aligned}
  &-\overline M_{a,b}'(x)P_{a,b}(x)
  +
  \underline M_{a,b}'(x)\overline P_{a,b}(x) \\
  &\quad =
  \frac{
    P_{a,b}(x)\overline{\mathrm P}_{a,b}^{1}(x)
    +
    \overline P_{a,b}(x)\underline{\mathrm P}_{a,b}^{1}(x)}
       {b^2x^2p_{a,b}(x)}.
\end{aligned}
\]
By \eqref{eq:P1-identities},
\[
  P_{a,b}(x)\overline{\mathrm P}_{a,b}^{1}(x)
  +
  \overline P_{a,b}(x)\underline{\mathrm P}_{a,b}^{1}(x)
  =
  bxp_{a,b}(x).
\]
Therefore
\[
  -\overline M_{a,b}'(x)P_{a,b}(x)
  +
  \underline M_{a,b}'(x)\overline P_{a,b}(x)
  =
  \frac1{bx}.
\]
Consequently,
\[
  g_{h;a,b}'(x)
  =
  \mathbb E\left[
    h'(X)K_{a,b}^{2,1}(X,x)
  \right]
  +
  \frac{h'(x)}{bx}
  =
  \ell_{h;a,b}(x).
\]

Assume now that \(h\in\mathcal Z_2\). From the preceding identity,
\[
  \ell_{h;a,b}(x)
  =
  -\overline M_{a,b}''(x)\underline\Phi_h(x)
  -
  \underline M_{a,b}''(x)\overline{\Phi}_h(x)
  +
  \frac{h'(x)}{bx}.
\]
We integrate by parts in the representation of
\(\widetilde\ell_{h;a,b}\). Since \(h\) is compactly supported,
\[
  \int_0^x \underline{\mathrm P}_{a,b}^{1}(t)h''(t)\,dt
  =
  \underline{\mathrm P}_{a,b}^{1}(x)h'(x)
  -
  \int_0^x P_{a,b}(t)h'(t)\,dt
\]
and
\[
  \int_x^\infty \overline{\mathrm P}_{a,b}^{1}(t)h''(t)\,dt
  =
  -\overline{\mathrm P}_{a,b}^{1}(x)h'(x)
  +
  \int_x^\infty \overline P_{a,b}(t)h'(t)\,dt .
\]
Therefore, using \eqref{eq:kernel-22},
\[
\begin{aligned}
  \widetilde\ell_{h;a,b}(x)
  &=
  \mathbb E\left[
    h''(X)K_{a,b}^{2,2}(X,x)
  \right] \\
  &=
  -\overline M_{a,b}''(x)\underline\Phi_h(x)
  -
  \underline M_{a,b}''(x)\overline{\Phi}_h(x) \\
  &\quad
  +
  h'(x)
  \left\{
    \overline M_{a,b}''(x)\underline{\mathrm P}_{a,b}^{1}(x)
    +
    \underline M_{a,b}''(x)\overline{\mathrm P}_{a,b}^{1}(x)
  \right\}.
\end{aligned}
\]
By Lemma~\ref{lma:boundskern}, or directly from
\eqref{eq:Msecond-identities} and \eqref{eq:P1-identities},
\[
  \overline M_{a,b}''(x)\underline{\mathrm P}_{a,b}^{1}(x)
  +
  \underline M_{a,b}''(x)\overline{\mathrm P}_{a,b}^{1}(x)
  =
  \frac1{bx}.
\]
Hence
\[
  \widetilde\ell_{h;a,b}(x)
  =
  \ell_{h;a,b}(x).
\]

It remains to differentiate this second-order representation. Write
\[
  \underline\Phi^1_h(x)
  :=
  \int_0^x \underline{\mathrm P}_{a,b}^{1}(t)h''(t)\,dt,
  \qquad
  \overline{\Phi}^1_h(x)
  :=
  \int_x^\infty \overline{\mathrm P}_{a,b}^{1}(t)h''(t)\,dt .
\]
Then
\[
  \ell_{h;a,b}(x)
  =
  \overline M_{a,b}''(x)\underline\Phi^1_h(x)
  -
  \underline M_{a,b}''(x)\overline{\Phi}^1_h(x).
\]
Since
\[
 {\underline\Phi^1_h}'(x)=\underline{\mathrm P}_{a,b}^{1}(x)h''(x),
  \qquad
 {\overline{\Phi}^1_h}'(x)=-\overline{\mathrm P}_{a,b}^{1}(x)h''(x),
\]
we get
\[
\begin{aligned}
  \ell_{h;a,b}'(x)
  &=
  \overline M_{a,b}'''(x)\underline\Phi^1_h(x)
  -
  \underline M_{a,b}'''(x)\overline{\Phi}^1_h(x) \\
  &\quad
  +
  h''(x)
  \left\{
    \overline M_{a,b}''(x)\underline{\mathrm P}_{a,b}^{1}(x)
    +
    \underline M_{a,b}''(x)\overline{\mathrm P}_{a,b}^{1}(x)
  \right\}.
\end{aligned}
\]
The first line is
\[
  \mathbb E\left[
    h''(X)K_{a,b}^{3,2}(X,x)
  \right],
\]
and the factor in braces is again \(1/(bx)\). Thus
\[
  \ell_{h;a,b}'(x)
  =
  \mathbb E\left[
    h''(X)K_{a,b}^{3,2}(X,x)
  \right]
  +
  \frac{h''(x)}{bx}
  =
  m_{h;a,b}(x).
\]

We have proved the differentiability statements and the identities for the
pointwise representatives. It remains to prove the boundedness assertions.
If \(h\in\mathcal Z_1\), then \(\|h'\|_\infty\le1\), and Lemma
\ref{lma:boundskern} gives
\[
  |f_{h;a,b}(x)|
  \le
  \mathbb E[|K_{a,b}^{0,1}(X,x)|]
  =
  1,
\]
\[
  |g_{h;a,b}(x)|
  \le
  \mathbb E[|K_{a,b}^{1,1}(X,x)|]
  =
  C_{a,b}^{1}(x),
\]
and
\[
  |\ell_{h;a,b}(x)|
  \le
  \mathbb E[|K_{a,b}^{2,1}(X,x)|]
  +
  \frac1{bx}
  =
  \frac2{bx}.
\]
Using the scaling relation \eqref{eq:M-scaling} and
Proposition~\ref{prop:Ca-bounds}, it follows that
\(f_{h;a,b}\), \(g_{h;a,b}\), \(xg_{h;a,b}\), and \(x\ell_{h;a,b}\) are
bounded.

If \(h\in\mathcal Z_2\), then \(\|h''\|_\infty\le1\), and the representation
\(\ell_{h;a,b}=\widetilde\ell_{h;a,b}\) gives
\[
  |\ell_{h;a,b}(x)|
  \le
  \mathbb E[|K_{a,b}^{2,2}(X,x)|]
  =
  C_{a,b}^{2}(x).
\]
Moreover,
\[
  |m_{h;a,b}(x)|
  \le
  \mathbb E[|K_{a,b}^{3,2}(X,x)|]
  +
  \frac1{bx}
  =
  \frac2{bx}.
\]
Again by \eqref{eq:M-scaling} and Proposition~\ref{prop:Ca-bounds},
\(\ell_{h;a,b}\) and \(xm_{h;a,b}\) are bounded.

Finally, the extension from smooth compactly supported \(h\) to
\(\mathcal Z_1\) and \(\mathcal Z_2\) follows by mollification and
truncation, using the same bounds and dominated convergence.
\end{proof}
{

\begin{proof}[Proof of Proposition~\ref{prop:Kol-small-shape}]
We follow the smoothing argument used in the proof of Corollary~1 of
\citet{Gaunt2022}. Write $G=G_{a,b}$ and $D=d_1(Y,G)$. Since $0<a\le1$,
the gamma density $p_{a,b}$ is nonincreasing on $(0,\infty)$. Thus, for
every $t>0$,
\[
  \sup_{x\in\mathbb R}\P[x<G\le x+t]
  = \P[0<G\le t]
  \le \frac{1}{\Gamma(a)b^a}\int_0^t y^{a-1}\,dy
  = \frac{(t/b)^a}{\Gamma(a+1)}.
\]

For $z\in\mathbb R$ and $w>0$, let $h_z(x)=\mathbf{1}_{\{x\le z\}}$ and set
\[
  h_z^\pm(x)=\frac{1}{w}\int_0^w h_{z\pm t}(x)\,dt.
\]
Then $h_z^-\le h_z\le h_z^+$, and both $h_z^-$ and $h_z^+$ are
$1/w$-Lipschitz. By mollification, the definition of $d_1$ therefore gives
\[
  \big|\E[h_z^\pm(Y)]-\E[h_z^\pm(G)]\big|\le \frac{D}{w}.
\]
Using $h_z^+$ for the upper bound and $h_z^-$ for the lower bound, we obtain
\begin{align*}
  \P[Y\le z]-\P[G\le z]
  &\le \frac{D}{w}+\frac{1}{w}\int_0^w \P[z<G\le z+t]\,dt,\\
  \P[G\le z]-\P[Y\le z]
  &\le \frac{D}{w}+\frac{1}{w}\int_0^w \P[z-t<G\le z]\,dt.
\end{align*}
In each integral, the interval has length $t$. Applying the preceding
bound and taking the supremum over $z$ gives
\begin{equation}\label{eq:Kol-every-w}
  d_{\mathrm{Kol}}(Y,G)
  \le \frac{D}{w}
      +\frac{1}{w}\int_0^w\frac{(t/b)^a}{\Gamma(a+1)}\,dt
  = \frac{D}{w}+\frac{(w/b)^a}{\Gamma(a+2)}.
\end{equation}
If $D=0$, the conclusion follows by letting $w\downarrow0$.
If $D>0$, the right-hand side of \eqref{eq:Kol-every-w} is minimized at
\[
  w_*=b\left(\frac{\Gamma(a+2)D}{ab}\right)^{1/(a+1)}.
\]
Substituting this value yields
\[
  d_{\mathrm{Kol}}(Y,G)
  \le
  \Gamma(a+1)^{-1/(a+1)}
  \left(\frac{a+1}{a}\right)^{a/(a+1)}
  \left(\frac{D}{b}\right)^{a/(a+1)},
\]
which proves \eqref{eq:Kol-small-shape}.
\end{proof}

}

\begin{prop}
\label{prop:Ca-bounds}
For every \(a>0\), the functions
\[
  z\mapsto C_{a,1}^{1}(z),
  \qquad
  z\mapsto zC_{a,1}^{1}(z),
  \qquad
  z\mapsto C_{a,1}^{2}(z)
\]
are bounded on \((0,\infty)\). Consequently,
\begin{equation}
\label{eq:C-D-E-def}
  c_a:=\sup_{z>0}C_{a,1}^{1}(z),
  \qquad
  d_a:=\sup_{z>0}zC_{a,1}^{1}(z),
  \qquad
  e_a:=\sup_{z>0}C_{a,1}^{2}(z)
\end{equation}
are finite. Moreover, as \(a\to\infty\),
\[
   \sqrt a\,c_a\to\sqrt{\frac{2}{\pi}},
  \qquad
  \frac{d_a}{\sqrt a}\to\sqrt{\frac{2}{\pi}},
  \qquad
  \sqrt a\,e_a\to\sqrt{\frac{\pi}{8}}.
\]
\end{prop}

\begin{proof}
Throughout the proof we take \(b=1\). Recall the notation
\eqref{eq:ij}. From \eqref{eq:Mprime-identities} and
\eqref{eq:Msecond-identities},
\[
  C_{a,1}^{1}(x)
  =
  2\frac{\underline{\mathrm P}^{1}_{a,1}(x)
  \overline{\mathrm P}^{1}_{a,1}(x)}{x^2p_{a,1}(x)},
  \qquad
  C_{a,1}^{2}(x)
  =
4  \frac{I_{a,1}(x)J_{a,1}(x)}{x^3p_{a,1}(x)}.
\]
If \(X\sim\mathcal G_{a,1}\), then, by Fubini,
\[
  \underline{\mathrm P}^{1}_{a,1}(x)=\E[(x-X)_+],
  \quad
  \overline{\mathrm P}^{1}_{a,1}(x)=\E[(X-x)_+],
  \quad
  I_{a,1}(x)=\frac12\E[(x-X)_+^2],
  \quad
  J_{a,1}(x)=\frac12\E[(X-x)_+^2].
\]
As \(x\downarrow0\), using
\(p_{a,1}(x)\sim x^{a-1}/\Gamma(a)\) and
\(P_{a,1}(x)\sim x^a/\Gamma(a+1)\), we get
\[
  \underline{\mathrm P}^{1}_{a,1}(x)
  \sim \frac{x^{a+1}}{\Gamma(a+2)},
  \quad
  \overline{\mathrm P}^{1}_{a,1}(x)\to a,
  \quad
  I_{a,1}(x)\sim \frac{x^{a+2}}{(a+2)\Gamma(a+2)},
  \quad
  J_{a,1}(x)\to \frac{a(a+1)}2.
\]
Thus
\[
  C_{a,1}^{1}(x)\to \frac2{a+1},
  \qquad
  xC_{a,1}^{1}(x)\to0,
  \qquad
  C_{a,1}^{2}(x)\to \frac2{a+2}.
\]
As \(x\to\infty\), \(\underline{\mathrm P}^{1}_{a,1}(x)-
\overline{\mathrm P}^{1}_{a,1}(x)=x-a\), while the standard
incomplete-gamma expansion gives
\(\overline P_{a,1}(x)\sim p_{a,1}(x)\) and
\(\overline{\mathrm P}^{1}_{a,1}(x)\sim p_{a,1}(x)\). Hence
\[
  C_{a,1}^{1}(x)\sim \frac2x,
  \qquad
  xC_{a,1}^{1}(x)\to2.
\]
Moreover \(I_{a,1}(x)\sim x^2/2\) and \(J_{a,1}(x)\sim p_{a,1}(x)\), so
\(C_{a,1}^{2}(x)\sim2/x\to0\). By continuity and the finite one-sided
limits just computed, the three functions \(C_{a,1}^{1}\),
\(xC_{a,1}^{1}\), and \(C_{a,1}^{2}\) are bounded on \((0,\infty)\), and
therefore \(c_a,d_a,e_a<\infty\).

We turn to the large-\(a\) behaviour. Put
\(x=a+t\sqrt a\) and \(Y_a=(X-a)/\sqrt a\). The local central limit theorem
gives, locally uniformly in \(t\),
\[
  \sqrt a\,p_{a,1}(a+t\sqrt a)\to\varphi(t).
\]
Since \(\sup_{a\ge1}\E[Y_a^4]<\infty\), the family
\((Y_a^2)_{a\ge1}\) is uniformly integrable, and hence, also locally
uniformly in \(t\),
\[
  \frac{\underline{\mathrm P}^{1}_{a,1}(a+t\sqrt a)}{\sqrt a}
  \to \E[(t-Z)_+]=t\Phi(t)+\varphi(t),
  \qquad
  \frac{\overline{\mathrm P}^{1}_{a,1}(a+t\sqrt a)}{\sqrt a}
  \to \E[(Z-t)_+]=\varphi(t)-t\overline\Phi(t),
\]
where \(Z\sim N(0,1)\). Consequently,
\[
  \sqrt a\,C_{a,1}^{1}(a+t\sqrt a)\to K_1(t),
  \qquad
  \frac{(a+t\sqrt a)C_{a,1}^{1}(a+t\sqrt a)}{\sqrt a}\to K_1(t),
\]
locally uniformly, with
\[
  K_1(t)=
  2\frac{\E[(t-Z)_+]\E[(Z-t)_+]}{\varphi(t)}.
\]
Similarly,
\[
  \frac{I_{a,1}(a+t\sqrt a)}{a}\to\frac12\E[(t-Z)_+^2],
  \qquad
  \frac{J_{a,1}(a+t\sqrt a)}{a}\to\frac12\E[(Z-t)_+^2],
\]
locally uniformly, and therefore
\[
  \sqrt a\,C_{a,1}^{2}(a+t\sqrt a)\to K_2(t),
  \qquad
  K_2(t)=
  \frac{\E[(t-Z)_+^2]\E[(Z-t)_+^2]}{\varphi(t)}.
\]

It remains only to pass from local convergence to convergence of the
suprema. We use the standard Gamma Mills estimates, obtained by integration
by parts at the endpoint. They imply that, for some universal constant
\(C\),
\[
\limsup_{a\to\infty}\sup_{|x-a|\ge R\sqrt a}
\max\left\{
\sqrt a\,C_{a,1}^{1}(x),
\frac{xC_{a,1}^{1}(x)}{\sqrt a},
\sqrt a\,C_{a,1}^{2}(x)
\right\}
\le C(R^{-1}+R^{-3}).
\]
Indeed, the smaller one-sided first and second truncated moments are
\(O(p_{a,1}(x)x^2/|x-a|^2)\) and
\(O(p_{a,1}(x)x^3/|x-a|^3)\), respectively. Thus no maximizer can escape the
window \(x=a+O(\sqrt a)\). Since also \(K_1(t),K_2(t)\to0\) as
\(|t|\to\infty\), the locally uniform convergence above yields
\[
  \sqrt a\,c_a\to \sup_{t\in\R}K_1(t),
  \qquad
  \frac{d_a}{\sqrt a}\to \sup_{t\in\R}K_1(t),
  \qquad
  \sqrt a\,e_a\to \sup_{t\in\R}K_2(t).
\]
A direct differentiation, using
\(\E[(t-Z)_+]=t\Phi(t)+\varphi(t)\) and
\(\E[(Z-t)_+]=\varphi(t)-t\overline\Phi(t)\), shows that \(K_1\) is even and
strictly decreasing on \((0,\infty)\). The analogous computation with the
second truncated moments gives the same conclusion for \(K_2\). Hence both
are maximized at \(0\). Since
\[
  K_1(0)=\sqrt{\frac2\pi},
  \qquad
  K_2(0)=\sqrt{\frac{\pi}{8}},
\]
the announced limits follow.
\end{proof}

\begin{prop}
\label{prop:ea-explicit}
For every \(a\ge17\),
\[
  e_a\le \frac{5}{2\sqrt a}.
\]
\end{prop}

\begin{proof}
We take \(b=1\) and write \(I=I_{a,1}\), \(J=J_{a,1}\),
\(p=p_{a,1}\). Let \(X\sim\mathcal G_{a,1}\). By Fubini,
\[
  I(x)=\frac12\E[(x-X)_+^2],
  \qquad
  J(x)=\frac12\E[(X-x)_+^2],
\]
so
\begin{equation}
  I(x)+J(x)=\frac12\{a+(x-a)^2\}.
  \label{eq:iplusj}
\end{equation}
Put \(n=a-1\). If \(0<x\le n\), set
\(\tau=(n-x)/\sqrt n\). With \(s=x-y\),
\[
  I(x)=\frac12\int_0^x s^2p(x-s)\,ds.
\]
For \(0<s<x\),
\[
  \frac{p(x-s)}{p(x)}
  =
  \left(1-\frac{s}{x}\right)^n e^s .
\]
Taking \(s=xz/\sqrt n\) and using \(x=n-\tau\sqrt n\), the inequality
\(\log(1-u)\le -u-u^2/2\) gives
\[
  \frac{p(x-s)}{p(x)}
  \le e^{-\tau z-z^2/2}.
\]
Thus
\[
  I(x)\le\frac{x^3p(x)}{2n^{3/2}}L(\tau),
  \qquad
  L(\tau):=\int_0^\infty z^2e^{-\tau z-z^2/2}\,dz .
\]
Since \(J(x)\le I(x)+J(x)\), \eqref{eq:iplusj} yields
\[
  \sqrt a\,C_{a,1}^{2}(x)
  \le
  L_n(\tau):=
  \frac{\sqrt{n+1}}{n^{3/2}}
  \{n+1+(\tau\sqrt n+1)^2\}L(\tau).
\]
Moreover \(L_n\) is decreasing: differentiating and using
\(L'(\tau)=-\int_0^\infty z^3e^{-\tau z-z^2/2}\,dz\), one obtains
\(L_n'(\tau)\le0\) after one integration by parts. Hence
\[
  L_n(\tau)\le L_n(0)
  =
  \frac{\sqrt{n+1}(n+2)}{n^{3/2}}\sqrt{\frac{\pi}{2}}
  \le\frac52
\]
for \(a\ge4\). Therefore
\[
  \sqrt a\,C_{a,1}^{2}(x)\le\frac52,
  \qquad 0<x\le a-1 .
\]

If \(x\ge n\), set \(\tau=(x-n)/\sqrt n\). With \(s=y-x\),
\[
  J(x)
  =
  \frac12\int_0^\infty s^2p(x+s)\,ds.
\]
Taking \(s=xz/\sqrt n\) and using \(x=n+\tau\sqrt n\),
\[
  J(x)=\frac{x^3p(x)}{2n^{3/2}}R_n(\tau),
\]
where
\[
  R_n(\tau):=
  \int_0^\infty z^2
  \exp\left\{
    n\log\left(1+\frac z{\sqrt n}\right)
    -(\sqrt n+\tau)z
  \right\}\,dz .
\]
Since \(I(x)\le I(x)+J(x)\), \eqref{eq:iplusj} gives
\[
  \sqrt a\,C_{a,1}^{2}(x)
  \le
  \widetilde R_n(\tau):=
  \frac{\sqrt{n+1}}{n^{3/2}}
  \{n+1+(\tau\sqrt n-1)^2\}R_n(\tau).
\]
Again \(\widetilde R_n\) is decreasing: differentiating under the integral
sign and integrating by parts in \(z\) gives
\(\widetilde R_n'(\tau)\le0\). Hence
\[
  \widetilde R_n(\tau)\le\widetilde R_n(0)
  =
  \frac{e^n(n+2){E_{-n-2}(n)}}{\sqrt{n+1}}
\]
 where
\[
  { E_\nu(z)}
  :=
  \int_1^\infty e^{-zt}t^{-\nu}\,dt .
\]
{
{
To see that $\widetilde R_n(0)$ is decreasing in $n$, observe that
\[
  n\log\left(1+\frac{z}{\sqrt n}\right)-\sqrt n\,z
  =-\int_0^z \frac{t}{1+t/\sqrt n}\,dt .
\]
For each $z>0$, the right-hand side decreases with $n$, so the
integral defining $R_n(0)$ also decreases with $n$. Since
\[
  \widetilde R_n(0)
  =\sqrt{1+\frac1n}\left(1+\frac2n\right)R_n(0)
\]
and the prefactor is positive and decreasing, it follows that
$\widetilde R_n(0)\le\widetilde R_{16}(0)<5/2$ for every $n\ge16$.
}}
\end{proof}

\end{document}